\documentclass{siamltex1213}
\usepackage{amssymb}
\usepackage{amsmath}
\usepackage{graphicx}
\usepackage{latexsym}
\usepackage{wrapfig}
\usepackage{placeins}
\usepackage{float}
\usepackage{caption}
\usepackage{subcaption}
\usepackage{cleveref}

\renewcommand{\div}{\operatorname{div}}

\newcommand{\RT}{\mathcal{RT}}

\newcommand{\jump}[1]{[\![ #1]\!]}
\newcommand{\gjump}[1]{[ #1]}

\newtheorem{remark}{Remark}

\renewcommand{\div}{{\hbox{div}}}

\newcommand{\vsf}{\vspace*{.4cm}}
\newcommand{\dx}{\, dx}
\newcommand{\ds}{\, ds}

\newcommand{\T}{\mathcal{T}_h}
\newcommand{\C}{\mathcal{C}_h}
\newcommand{\D}{\mathcal{D}_h}
\newcommand{\M}{\mathcal{M}_h}
\newcommand{\F}{\mathcal{F}_h}
\newcommand{\N}{\mathcal{N}_h}

\newcommand{\nrmf}[1]{\|#1\|_{\mathcal{M}_h}}

\newcommand{\Ker}{\mathrm{Ker}\, }

\newtheorem{assumption}{Assumption}
\renewcommand{\div}{{\hbox{div}}}

\newcommand{\IRT}{\mathcal{IRT}}

\newcommand{\Fg}{\mathcal{F}_g}

\newcommand{\nrmL}[2][1]{\|#2\|_{#1}}

\newcommand{\nrmD}[1]{\|#1\|_{\mathcal{D}_h}}

\newtheorem{example}{Example}[section]

\title{Elliptic interface problem approximated by CutFEM: I. Conservative flux recovery and numerical validation of adaptive mesh refinement}
\author{Daniela Capatina\footnotemark[1] \and Aimene Gouasmi\footnotemark[1] \and Cuiyu He\footnotemark[2]}
\begin{document}
%
%
\maketitle
\renewcommand{\thefootnote}{\fnsymbol{footnote}}
\footnotetext[1]{LMAP CNRS UMR 5142, University of Pau, 64013 Pau, France}
\footnotetext[2]{The University of Georgia, Athens, GA 30602, USA}

\renewcommand{\thefootnote}{\arabic{footnote}}
\slugger{sinum}{xxxx}{xx}{x}{x--x}

\begin{abstract}
We study an elliptic interface problem with discontinuous diffusion coefficients on unfitted meshes using the CutFEM method. Our main contribution is the reconstruction of conservative fluxes from the CutFEM solution and their use in a posteriori error estimation. We introduce a hybrid mixed formulation with locally computable Lagrange multipliers and reconstruct the flux in the immersed Raviart-Thomas space. Based on this, we propose a new a posteriori error estimator that includes both volume and interface terms. We state its robust reliability and local efficiency, and validate the approach through numerical experiments.
\end{abstract}

\begin{keywords}
Elliptic interface problem; CutFEM; Conservative flux reconstruction; A posteriori error estimation; Adaptive finite element methods; Immersed Raviart-Thomas elements

\end{keywords}

\begin{AMS}
65N12, 65N15, 65N30, 65N50, 65N85
\end{AMS}

\pagestyle{myheadings}
\thispagestyle{plain}
\markboth{}{Daniela Capatina, Aimene Gouasmi and Cuiyu He}
%
\section{Introduction}\label{sec:Introduction}
%
Local reconstruction of conservative fluxes from finite element solutions plays a key role in applications such as a posteriori error estimation \cite{Ai:07b, braess2008equilibrated, Ve:09, BFH:14, ErnVo2015, CaCaZh:20, cai2021generalized}  and enforcing flux conservation in continuum mechanics \cite{odsaeter2017postprocessing, ern2009accurate, vohralik2013posteriori}. These techniques are essential for accurately representing physical fluxes -such as heat, mass, or momentum- in solid mechanics and porous media.

In this work, we consider an elliptic problem with an interface that is not aligned with the finite element mesh, characterized by discontinuous diffusion coefficients and standard transmission conditions at the interface; we allow for a jump in the normal flux across the interface. We use the CutFEM method (cf.~\cite{CutFEM2, CutFEM}) to solve the interface problem -a Nitsche type formulation that is robust with respect to discretization, diffusion contrasts, and mesh/interface geometry, owing to the inclusion of additional stabilization terms.

Our goal is to reconstruct conservative fluxes and employ them in a posteriori error analysis and adaptive mesh refinement. To the best of our knowledge, these topics remain largely unexplored in the context of CutFEM. While conservative flux reconstruction for CutFEM solutions has been investigated in~\cite{capatina_he2021flux} for Poisson boundary problems on unfitted meshes, no such developments exist for interface problems.

Regarding the reconstruction method, we have chosen to generalize an approach previously developed for the Poisson problem in \cite{Dana2016} and then extended to diffusion problems in \cite{Aimene}.
The key idea is to introduce a hybrid mixed formulation whose primal solution is equivalent to the discrete CutFEM solution, while also incorporating an additional Lagrange multiplier defined on mesh edges to correct the normal trace of the flux. It should be noted that we do not solve any mixed problem for the reconstruction, contrary to other existing techniques such as \cite{vohralik2011guaranteed}. The Lagrange multiplier can be computed locally on element patches.

The extension of the approach developed in~\cite{Dana2016, Aimene} to elliptic interface problems using CutFEM on unfitted meshes raises several important questions.

The first question concerns the well-posedness of the hybrid mixed formulation. To ensure the stability, convergence, and robustness of the numerical method, the constants involved in the analysis ideally should be independent of the discretization parameters, the diffusion coefficients, and the mesh/interface geometry. Following the approach in~\cite{Aimene} for fitted meshes, we propose a mixed formulation with Lagrange multipliers assigned separately to the edges within each individual subdomain. In the case of cut elements, this results in two distinct multipliers defined over each cut edge that has non-zero intersections with all subdomains. We demonstrate that the mixed formulation possesses several key properties: its primal solution is equivalent to the original discrete CutFEM solution, it is fully robust with respect to both numerical and physical parameters, and it allows the multipliers to be computed locally.

The second question concerns the reconstruction of numerical fluxes in cells cut by the interface, where appropriate notions of discrete conservation and transmission conditions must be defined. Although multiple definitions of fluxes are possible, this issue is closely tied to a third question: the development of a posteriori error estimators based on equilibrated fluxes. We propose a global error estimator consisting of two parts: a standard term -the weighted $L^2$-norm of the difference between the equilibrated and numerical fluxes -and a new interface term that accounts for discontinuities in the approximate solution across $\Gamma$. The main challenge lies in proving both the (sharp) reliability and the local efficiency of the estimator, with constants that remain robust with respect to the diffusion coefficients and the mesh-interface configuration.

To enforce both local conservation and normal trace continuity in the cut cells, we define a unique flux \(\sigma_h\) satisfying \([\sigma_h \cdot n_{\Gamma}] = 0\) across the interface \(\Gamma\) and \(\div \sigma_h = -f_h\). The Lagrange multipliers introduced earlier serve to correct the normal trace of the numerical flux from the CutFEM solution. For the sake of local efficiency in the a posteriori error analysis, the flux reconstruction is performed in the immersed Raviart-Thomas space recently introduced in~\cite{IRT}.

The immersed finite element (IFE) method \cite{li1998immersed} aims to modify traditional finite element spaces in order to recover optimal approximation capabilities on unfitted meshes. Notably, the IFE method retains the same degrees of freedom as traditional finite element methods and can revert to the conventional finite element method when the interface is absent. This characteristic, where IFE spaces are isomorphic to standard finite element spaces defined on the same mesh, is particularly beneficial for problems involving moving interfaces \cite{guo2021solving}.

The lowest-order immersed Raviart-Thomas space \(\IRT^0\)~\cite{guo2021solving, IRT} was developed to handle unfitted meshes by modifying standard Raviart-Thomas functions~\cite{raviart1977mixed} to maintain optimal approximation properties on cut elements. Functions in \(\IRT^0\) are piecewise \(\mathcal{RT}^0\), enforce strong continuity of the normal trace across the interface, and incorporate weak continuity of tangential flux. However, they only satisfy weak continuity on cut edges, and thus \(\IRT^0\) is not conforming in \(H(\div, \Omega)\), which introduces an additional a posteriori error term on the cut edges. The method also accommodates non-homogeneous transmission conditions.

A detailed theoretical analysis -presented in~\cite{Article2}- establishes the robust reliability and local efficiency of the proposed a posteriori error estimator, with constants that depend explicitly on the diffusion coefficients and the mesh/interface configuration. In the present paper, we summarize the main results of this analysis and validate them through a series of numerical experiments.

The paper is organized as follows. The model problem and relevant notation are introduced in Section~\ref{sec:Continuous_pb}. Section~\ref{sec:Discrete_pb} presents the finite element discretization on unfitted meshes using CutFEM and sets the foundation for flux recovery via an equivalent mixed formulation with locally computable Lagrange multipliers. In Section~\ref{sec:flux}, we describe the local flux reconstruction in the immersed Raviart-Thomas space and establish the local conservation property. Section~\ref{sec: A posteriori_IRT} applies the reconstructed fluxes to the a posteriori error analysis, where we define error estimators and state their sharp reliability and local efficiency. Section~\ref{sec:num_sim} reports several numerical experiments that confirm the theoretical results. The paper concludes with an appendix detailing the numerical implementation of the immersed Raviart-Thomas space.

\section{The continuous problem and notation}\label{sec:Continuous_pb}

Let $\Omega$ be a 2D polygonal domain and $\Gamma$ a sufficiently smooth interface separating $\Omega$ into two disjoint subdomains:$\bar \Omega=\bar{\Omega}^{1}\cup \bar{\Omega}^{2}$, $\partial\Omega^{1}\cap \partial\Omega^{2}=\Gamma$. We denote by $n_{\Gamma}$ the unit normal vector to $\Gamma$ oriented from $\Omega^1$ to $\Omega^2$.
We consider the following  model problem:
\begin{equation}\label{eq: continuous_problem_weak}
	\left\{
	\begin{aligned}
		-div(K\nabla u)=f \quad &\text{in }\Omega^i\,\, (i=1,2),\\
		u=0 \quad & \text{on }\partial\Omega \\ 
		\gjump{u}=0,\,\,\gjump{K\nabla u\cdot n_{\Gamma}}=g  \quad &\text{on }\Gamma.
	\end{aligned}
	\right.
\end{equation}
The jumps across $\Gamma$ are given by 
$$\gjump{u}=u_1-u_2,\quad \gjump{K\nabla u\cdot n_{\Gamma}}=(K_1\nabla u_1-K_2\nabla u_2)\cdot n_{\Gamma},$$ 
where $u_{|\Omega^i}=u_i$ and $K_{|\Omega^i}=K_i$, for $i=1, 2$. We suppose $f\in L^2(\Omega)$, $g \in L^2(\Gamma)$ and, for the sake of simplicity, here we assume $K_i=k_i I$ with $k_i>0$, for $i=1,2$. The approach can be extended to piecewise constant positive definite tensors $K$ and to other boundary conditions on $\partial \Omega$. 

We consider the following weak formulation of problem \eqref{eq: continuous_problem_weak}, which clearly has a unique solution: Find $u\in H^1_0(\Omega) $ such that

\begin{equation}\label{eq: continuous_problem}
\int_{\Omega}K\nabla u \cdot \nabla v\dx= \int_{\Omega}f v\dx+\int_{\Gamma}gv\ds\quad \forall v\in H^1_0(\Omega).
\end{equation}  

We next introduce some notation for the finite element approximation of \eqref{eq: continuous_problem}. Let $\mathcal{T}_h$ denote a regular triangular mesh of $\Omega$, whose elements are closed sets, and $\mathcal{F}_h$ be the set of edges. The diameter of $T \in \mathcal{T}_h$ (and the length of $F \in \mathcal{F}_h$) is denoted $h_T$ (and $h_F$). For an interior edge $F$, $n_F$ is a fixed unit normal vector to $F$, oriented from $T_F^-$ to $T_F^+$, where $T_F^-$ and $T_F^+$ are the two triangles that share $F$. If $F \subset \partial \Omega$, then $n_F$ is the outward normal vector to $\Omega$ while if $F \subset \Gamma$, then $n_F = n_{\Gamma}$. For $\omega \subset \mathbb{R}^d$ with $d = 1,2 $, we denote the $L^2(\omega)$-norm by $\|\cdot\|_\omega$ and the $L^2(\omega)$-orthogonal projection onto $P^{{m}}(\omega)$ by $\pi^{{m}}_{\omega}$, for ${m} \in \mathbb{N}$. For $i = 1, 2$, we define:
\begin{equation*}
	\mathcal{T}_h^{i} = \big\{ T \in \mathcal{T}_h \ ; \ T \cap \Omega^{i} \neq \emptyset \big\}, \quad \mathcal{F}_h^{i} = \big\{ F \in \mathcal{F}_h \ ; \ F \cap \Omega^{i} \neq \emptyset \big\}
\end{equation*}
and we set $\Omega_h^i = \displaystyle \bigcup_{T \in \mathcal{T}_h^i} T$; note that $\Omega^i \subset \Omega_h^i$. Let also $\mathcal {N}_h^i$ the set of nodes belonging to $\Omega_h^i$. As regards the cut elements, let:
\begin{equation*}
	\begin{split}
		\mathcal{T}_h^{\Gamma} = \big\{ T \in \mathcal{T}_h \ ; \ T \cap \Gamma \neq \emptyset \big\}, & \quad \mathcal{F}_h^{\Gamma} = \big\{ F \in \mathcal{F}_h \ ; \ F \cap \Gamma \neq \emptyset \big\}, \\
		\mathcal{F}_g^i = \big\{ F \in \mathcal{F}_h^i \ ; \ (T_F^+ \cup T_F^-) \cap \Gamma \neq \emptyset \big\}, & \quad T^i = T \cap \Omega^i \,\,\, \forall T \in \mathcal{T}_h^{\Gamma}, \quad i = 1,2.
	\end{split}
\end{equation*}

In order to focus on flux reconstruction, we assume here that $\Gamma$ is a polygonal line such that for each $T \in \mathcal{T}_h^{\Gamma}$, the intersection $\Gamma_T := T \cap \Gamma$ is a line segment. For a function $v \in L^{2}(\Omega)$, sufficiently smooth on each $\Omega^i$ but discontinuous across $\Gamma$, 
we define the two following means at a point $x \in \Gamma$:
\begin{equation*}
	\{v\}(x) = \omega_1v_1(x) + \omega_2v_2(x), \quad \{v\}^*(x) = \omega_2v_1(x) + \omega_1v_2(x),
\end{equation*}
where the weights $\omega_1$, $\omega_2$ are given (cf. \cite{Ern}) by:
\begin{equation*}
	\omega_1 = \frac{k_2}{k_2 + k_1}, \quad \omega_2 = \frac{k_1}{k_1 + k_2}.
\end{equation*}
It is also useful to introduce the harmonic mean $k_{\Gamma}=k_1 k_2/(k_1 + k_2)$.

Furthermore, we introduce the arithmetic mean and the jump across an interior edge $F \in \mathcal{F}_h^i$, for $1\leq i\leq 2$, as follows:
$$\langle v \rangle = \frac{1}{2} (v^- + v ^+), \quad \jump{v} = v^{-} - v^{+},\quad \jump{\partial_n v} =\jump{\nabla v} \cdot n_F.$$
For boundary edges, we set $\langle v \rangle = \jump{v} = v$. Finally, for $i = 1, 2$, let
\begin{equation*}\label{eq: v_{h,i}}
    V^i = \big\{ v \in H^1(\Omega_h^i) \ ; \ v_{|(\partial\Omega^i \setminus \Gamma)} = 0 \big\}.
\end{equation*} 

We use the symbols $\gtrsim $ and $\lesssim$ to indicate the existence of a generic constant that is independent of the mesh size, the interface geometry and the diffusion coefficients.

\section{The CutFEM approximation}\label{sec:Discrete_pb}

In the numerical approximation of \eqref{eq: continuous_problem}, the transmission conditions on $\Gamma$ are taken into account by means of Nitsche's method \cite{Nitche}. Moreover, we use CutFEM \cite{CutFEM2} to stabilize the approach with respect to the geometry of the interface, by adding a ghost penalty term. 

\subsection{Primal discrete formulation}
We consider the discrete space $\C=\C^1\times \C^2$, where for $1\le i\le 2$,
\begin{equation*}
  \C^i=\big\{ v \in V^i:\,  v_{|T}\in P^1(T),\,\, \forall \ T\in  \mathcal{T}_h^{i} \big\}.
\end{equation*}
Note that the degrees of freedom on the cut cells are doubled.     
We define the following bilinear and linear forms: for $u_h=(u_{h,1},u_{h,2})\in \C$ and $v_h=(v_{h,1},v_{h,2}) \in \C$, let
\begin{equation*}
	\begin{split}
		a_i(u_{h,i},v_{h,i})=&\sum_{T\in\T^i}\int_{T^{i}}k_i\nabla u_{h,i} \cdot \nabla v_{h,i}\dx=\int_{\Omega^i}k_i\nabla u_{h,i} \cdot \nabla v_{h,i}\dx\quad (i=1,\,2),\\
        j_i(u_{h,i},v_{h,i})=&\sum_{F\in\mathcal{F}_g^i} h_F \int_{F}k_i\jump{\partial_n u_{h,i}}\jump{\partial_n v_{h,i}}\ds\quad (i=1,\,2),\\
		a_{\Gamma}(u_h,v_h)=& \sum_{T\in\T^{\Gamma}}\int_{\Gamma_T}\bigg ( \frac{\gamma k_{\Gamma}}{h_T}[u_h][v_h]-\{K\nabla u_h\cdot n_{\Gamma}\}[v_h]-\{K\nabla v_h\cdot n_{\Gamma}\}[u_h]\bigg)\ds,\\
		l_h(v_h)=& \sum_{i=1}^{2}\int_{\Omega^i}f v_{h,i}\dx+\sum_{T\in\mathcal{T}_h^{\Gamma}}\int_{\Gamma_T}g\{v_h\}^*\ds.
	\end{split}
\end{equation*}
Here above, $a_i(\cdot,\cdot)$ and $j_i(\cdot,\cdot)$ represent the main part and the ghost penalty term, respectively, whereas the remaining form $a_{\Gamma}(\cdot.\cdot)$ takes into account the terms which result from the integration by parts, the symmetrization and the Nitsche stabilization. The discrete problem then reads: Find $u_h=(u_{h,1},u_{h,2}) \in \C $ such that
\begin{equation}\label{eq: Cutfem_Formulation}
	a_h(u_h,v_h)=l_h(v_h) \quad \forall v_h\in \C,
\end{equation}
where 
\begin{equation*}\label{eq: Cutfem_Formulation_form}
a_h(u_h,v_h)=\displaystyle{\sum_{i=1}^{2}}\bigg(a_i(u_{h,i},v_{h,i})+\gamma_g j_i(u_{h,i},v_{h,i})\bigg)+a_{\Gamma}(u_h,v_h).	
\end{equation*}
It is known that the stabilization parameters $\gamma >0$ and $\gamma_g>0$ can be chosen independently of the mesh, the interface geometry and the diffusion coefficients. 

For any $v_h\in \C$, we define the norm:
\begin{equation*}\label{eq: norm_h}
\|v_h\|_h^2= \sum_{i=1}^{2}\bigg(\|k_i^{1/2}\nabla v_{h,i}\|^2_{\Omega^i}+j_i(v_{h,i},v_{h,i})\bigg)
+\sum_{T\in \mathcal{T}_h^{\Gamma}}\int_{\Gamma_T}\frac{k_{\Gamma}}{h_T}[v_h]^2 \ds.
\end{equation*}

It is well-known that for $\gamma$ large enough, the bilinear form $a_h(\cdot,\cdot)$ is uniformly $(\C,\|\cdot\|_h)$-coercive. To establish this result, two lemmas are required.  
 \begin{lemma}[Trace Inequality]
     There exists a constant $C_\Gamma>0$  independent of the mesh/interface geometry  such that for any $T\in \T^\Gamma$ and  any $v\in H^1(T)$,
     \begin{equation}\label{eq: trace_in_cut}
         \nrmL[\Gamma_T]{v}^2\leq C_\Gamma\bigg( h_T^{-1}\nrmL[T]{v}^2+ h_T\nrmL[T]{\nabla v}^2\bigg).
     \end{equation}
 \end{lemma}

 \begin{lemma}\label{lem: bound_norms_cut}
      There exists a constant $C_g>0$ independent of the mesh-interface intersection for positive parameter $\gamma_g$, such that, for any $v_{h,i}\in \C^i$ ($i=1,2)$,
      \begin{equation*}
          C_g\nrmL[\Omega_h^i]{k_i^{1/2}\nabla v_{h,i}}^2\leq \nrmL[\Omega^i]{k_i^{1/2}\nabla v_{h,i}} ^2+\gamma_g j_i(v_{h,i},v_{h,i}).
      \end{equation*}
 \end{lemma}
For the proof, we refer the reader to \cite{CutFEM}.

Thanks to the Lax-Milgram theorem, the problem \eqref{eq: Cutfem_Formulation} has a unique solution. 
 
The following a priori error estimate is also known  \cite{CutFEM}.
\begin{theorem} \label{lem: a priori_cutfem}
Let $u_h$ be the CutFEM solution of \eqref{eq: Cutfem_Formulation} and $ u$  the solution of \eqref{eq: continuous_problem}. Assume that \( u _{|\Omega^i}\in H^{1+\varepsilon}(\Omega^i) \), where \(\varepsilon > 0\).
Then, there exists a constant \( C > 0 \) independent of the mesh size $h =\max_{T\in \mathcal T_h}h_T$ such that:
\begin{equation*}
\sum_{i=1}^2 \|u - u_{h,i}\|_{H^1(\Omega^i)} \leq C h^\varepsilon \sum_{i=1}^2\|u\|_{H^{1+\varepsilon}(\Omega^i)}.
\end{equation*}
\end{theorem}

\subsection{Equivalent mixed formulation}
In the sequel, we aim to construct, for $i\in \{1,2\}$, discrete functions $\theta_{h,i}$ living on the interior edges of the subdomains $\Omega_h^i$. These functions, which will serve to correct the normal flux, are introduced, following \cite{Dana2016}, as the Lagrange multipliers of a hybrid mixed formulation whose primal solution coincides with $u_h$. In the CutFEM context where one deals with cut edges, it is important to note that the multipliers are defined on the whole edges, in order to avoid the use of sub-edges which could be small. We thus employ standard finite element spaces on the whole elements. 

In order to introduce an equivalent mixed formulation of (\ref{eq: Cutfem_Formulation}), we first define the following finite-dimensional spaces:  
\begin{equation*}
    \D=\D^1\times \D^2,\quad \M=\M^1\times \M^2,
\end{equation*}
where for $i=1,2$ we set:
\begin{equation*}
	\begin{split}
		\D^i=&\{v\in L^2(\T^i);\, {v}_{|T}\in P^1(T)\ \forall T\in \T^i\},\\
		\M^i=&\bigg\{\mu\in L^2(\F^i);\, {\mu}|_{F}\in P^1(F)\,\, \forall F\in \F ^i,\, 
		\displaystyle\sum_{F\in\mathcal{F}_N}\mathfrak{s}_N^Fh_F{\mu}_{|F}(N)=0 \ \ \,\forall N\in \overset{\circ}{\mathcal {N}_h^i}\bigg \}.
	\end{split}     
\end{equation*}
Here above, $\overset{\circ}{\mathcal {N}_h^i}$ denotes the set of nodes interior to $\Omega_h^i$, $\mathcal{F}_N$ the set of edges sharing the node $N$, and $\mathfrak{s}_N^F= \text{sign}(n_F,N)$ is equal to $1 (-1)$ if the orientation of $n_F$ with respect to the the node $N$ is in the clockwise (counter-clockwise) rotation.

We endow these spaces with the following norms:
\begin{equation*}
	\begin{split}
		\nrmD{v_h}^2&=\|v_h\|_h^2+ \sum_{i=1}^{2}\sum_{F\in \F^i}h_F^{-1}\int_F k_i\jump{v_{h,i}}^2\ds,\quad \forall v_h=(v_{h,1},v_{h,2})\in \D,\\
		\nrmf{\mu_h}^2 &=\sum_{i=1}^{2}\sum_{F\in \F^i}h_F\int_{F}k_i (\mu_{h,i})^2\ds,\quad \forall \mu_h=(\mu_{h,1},\mu_{h,2}) \in \M.      
	\end{split}
\end{equation*}

We next consider the mixed formulation: Find $(\tilde{u}_h,\theta_h)\in \D\times \M$, such that 
\begin{equation}\label{eq: Mixed formulation_cut}
	\begin{split}
		\tilde {a}_h(\tilde{u}_h,v_h)+&b_h(\theta_h,v_h)=l_h(v_h)\quad \forall v_h \in \mathcal{D}_h,\\
		&b_h(\mu_h,\tilde{u}_h)=0 \quad \qquad\forall \mu_h \in \mathcal{M}_h,
	\end{split}
\end{equation}
where 
\begin{equation*}
    \begin{split}
        \tilde{a}_h(\cdot,\cdot)= &a_h(\cdot,\cdot)- d_h(\cdot,\cdot),\quad 
        b_h(\mu_h,v_h)=\displaystyle{\sum_{i=1}^{2}} b_h^i(\mu_{h,i},v_{h,i}),\\
		d_h(\tilde u_h,v_h)=&\sum_{i=1}^{2} \sum_{F\in\mathcal{F}_h^{i}}  \int_{F\cap\Omega^i} \bigg(\langle k_i\nabla \tilde {u}_{h,i}\cdot n_F\rangle \jump{v_{h,i}}  +\langle k_i\nabla v_{h,i} \cdot n_F\rangle \jump{\tilde{u}_{h,i}}\bigg)\ds,\\
		b_h^i(\mu_{h,i},v_{h,i})=&\sum_{F\in\mathcal{F}_h^{i}} \frac{k_i h_F}{2}\sum_{N\in\mathcal{N}_F}{\mu_{h,i}}_{|F}(N)\jump{v_{h,i}}(N) 
	\end{split}
\end{equation*} 
where $\mathcal{N}_F$ stands for the set of nodes belonging to the edge $F$. Note that in the definition of $ b_h^i(\cdot,\cdot) $, the trapezium formula (i.e. the Gauss-Lobatto quadrature formula with $2$ points) is used to approximate $\displaystyle   \int_{F}k_i\mu_{h,i}\jump{v_{h,i}}\ds $. 

\begin{lemma}\label{lem: kernel_b_h_cut}
	The discrete kernel of $b_h(\cdot,\cdot)$ coincides with the space $\C$, i.e.,
	\begin{equation*}
		\mathrm{Ker}\,b_h=\left\{ v_h\in\mathcal{D}_h;\, b_h(\mu_h,v_h)=0,\, \forall \mu_h \in \mathcal{M}_h \right\}=\mathcal{C}_h.
	\end{equation*}    
\end{lemma}

\begin{proof}
Let first $v_h\in \C$; the continuity of $v_h$ across any interior edge gives $b_h(\mu_h,v_h)=0$ for any $\mu_h \in \mathcal{M}_h$, which translates into $v_h\in \Ker b_h$. Next, we prove the remaining inclusion, $ \Ker b_h\subset \C$.  
For $v_h=(v_{h,1},v_{h,2})\in \Ker b_h$, we consider
$$(\mu_{h,i})_{|F}=h_F^{-1}\jump{v_{h,i}},\quad \forall F\in\F^i \quad (i=1,2).$$
Clearly, $\mu_{h,i} $ belongs to $\M^i$ since 
$$ \displaystyle\sum_{F\in\mathcal{F}_N}\mathfrak{s}_N^Fh_F(\mu_{h,i})_{|F}(N)=\displaystyle\sum_{F\in\mathcal{F}_N}\mathfrak{s}_N^F\jump{v_{h,i}}(N)=0, \quad \forall N\in \overset{\circ}{\mathcal {N}_h^i}.$$
From $b_h(\mu_h,v_h)=0$ we get $\jump{v_{h,i}}=0$ for any $F\in \F^i$ ($i=1,2$), which yields $v_h\in \C$.
The double inclusion yields the announced result.  
\end{proof}

We next establish the inf-sup condition for the form $b_h(\cdot,\cdot)$, which holds uniformly with respect to the discretization, the interface, and the diffusion coefficients. 
  
\begin{theorem}\label{thrm: inf-sup_cond_cut}
There exists a constant $\beta>0$ independent of the mesh size, $\Gamma$ and $K$ such that
	\begin{equation*}
		\inf_{\mu_h\in \mathcal{M}_h}\sup_{v_h\in \mathcal{D}_h} \frac{b_h(\mu_h,v_h)}{\nrmf{\mu_h}\nrmD{v_h}}\geq \beta.
	\end{equation*}
\end{theorem}
\begin{proof}
We follow the idea of \cite{Aimene} for the case of a diffusion problem where the discontinuities of the coefficients are aligned with the mesh, and adapt it here to the CutFEM method. The diffusion coefficients are supposed here to be constant on each subdomain; hence, the analysis provided in \cite{Dana2016} for the Laplace operator also applies. 

We use Fortin's trick \cite{brezzi2012mixed}: to any $\mu_h\in \M$, we associate a function $v_h\in \D$ s.t. 
\begin{equation}\label{eq: inf-sup bound_idea}
 \begin{split}
 b_h(\mu_h,v_h)\gtrsim \nrmL[\M]{\mu_h}^2,
 \quad \nrmL[\D]{v_h}\lesssim \nrmL[\M]{\mu_h}.
 \end{split}
\end{equation}

Let $\mu_h=(\mu_{h,1},\mu_{h,2})\in \M$. The construction of $v_h=(v_{h,1},v_{h,2})$ is done patch-wise. We fix $i\in\{1,2\}$, consider an arbitrary node $N\in\N^i$ 
 and denote by $\omega_N^i$ the patch consisting of the triangles of $\T^i$ that share node $N$. We define $v_{N,i}$ locally on $\omega_N^i$, piecewise linear and discontinuous, by imposing on any edge $F\in \F^i\cap \mathcal F_N$ that
\begin{equation}\label{eq:def_v}
\jump{v_{N,i}}_{|F}(N)=h_F(\mu_{h,i})_{|F}(N).
\end{equation}
We also impose $v_{N,i} (M)=0$ at all other vertices $M$ of the patch.
The linear system (\ref{eq:def_v}) is compatible for any $N \in \overset{\circ}{\mathcal {N}_h^i} $ due to the constraint imposed in the space $\mathcal M_h^i$. 
 It is also easy to check that the system is compatible when 
$N \in \N^i \setminus \overset{\circ}{\mathcal {N}_h^i}$.

Then we define $v_{h,i}\in \D^i$ as $\displaystyle v_{h,i}=\sum_{N\in \N^i} v_{N,i}$, which yields $\jump{v_{h,i}}_{|F}=h_F(\mu_{h,i})_{|F}$ for any $ F\in \F^i$. Hence, we further get that 
\begin{equation}\label{eq: inf_sub_lerror_1}    
b_h(\mu_h,v_h)= \sum_{i=1}^2\sum_{F\in\mathcal{F}_h^{i}} \frac{k_i h_F^2}{2}\sum_{N\in\mathcal{N}_F}(\mu_{h,i})_{|F}^2(N)
\gtrsim \nrmL[\M]{\mu_h}^2
\end{equation}
as well as: 
\begin{equation}\label{eq: inf_sub_lerror_2}
 \sum_{i=1}^{2}\sum_{F\in \F^i}\int_F k_i h_F^{-1}\jump{v_{h,i}}^2\ds =\sum_{i=1}^{2}\sum_{F\in \F^i}\int_F k_i h_F(\mu_{h,i})^2\ds= \nrmL[\M]{\mu_h}^2.
\end{equation}

It was shown in \cite{Aimene} that (\ref{eq:def_v}) admits a solution such that
\begin{equation}\label{eq:estim_v_L2}
    \sum_{T\in\mathcal T_h^i}h_T^{-2}\nrmL[T]{k_i^{1/2} v_{h,i}}^2 \lesssim  \nrmL[\M^i]{\mu_{h,i}}^2 \quad (i=1,2).
\end{equation}
By means of an inverse inequality and using that $\Omega^i\subset \Omega^i_h $, the previous bound implies
\begin{equation}\label{eq: inf_sub_lerror_3}
\sum_{i=1}^2 \nrmL[\Omega^i]{k_i^{1/2}\nabla v_{h,i}}^2\le \sum_{i=1}^2 \nrmL[\Omega^i_h]{k_i^{1/2}\nabla v_{h,i}}^2 \lesssim  \nrmL[\M]{\mu_h}^2.
\end{equation}

In what follows, we bound the remaining terms in the norm $\nrmL[\D]{v_h}$, which are specific to the CutFEM formulation. We clearly have:
\begin{equation*}
    j_i(v_{h,i},v_{h,i}) \le \sum_{F\in\Fg^i} k_i h_F (\nrmL[F]{\nabla v_{h,i}^+}^2+\nrmL[F]{\nabla v_{h,i}^-}^2)
 \lesssim \nrmL[\Omega_h^i]{ k_i^{1/2}\nabla v_{h,i}}^2\quad (i=1,2),
\end{equation*}
so the ghost penalty contribution is uniformly bounded:
\begin{equation}\label{eq: inf_sub_lerror_4}
    \sum_{i=1}^{2}j_i(v_{h,i},v_{h,i}) \lesssim  \nrmL[\M]{\mu_h}^2.
\end{equation}
Thanks to the trace inequality \eqref{eq: trace_in_cut}, combined with norm equivalence in finite dimensional spaces and with $k_\Gamma \leq k_i$ for $i=1,2$, one has  for any $T\in \T^\Gamma$ that:
\begin{equation*}
    \int_{\Gamma_T}\frac{k_{\Gamma}}{h_T}[v_h]^2 \ds \lesssim  \frac{k_{\Gamma}}{h_T^2}(\nrmL[T]{v_{h,1}}^2+\nrmL[T]{v_{h,2}}^2)\lesssim \frac{1}{h_T^2}(\nrmL[T]{k_1^{1/2}v_{h,1}}^2+\nrmL[T]{k_2^{1/2}v_{h,2}}^2).
\end{equation*}
Using again  \eqref{eq:estim_v_L2}, we end up with 
\begin{equation}\label{eq: inf_sub_lerror_5}
\begin{split}
    \sum_{T\in\T^{\Gamma}}\int_{\Gamma_T}\frac{k_{\Gamma}}{h_T}[v_h]^2 \ds \lesssim  \nrmL[\M]{\mu_h}^2.
\end{split}
\end{equation}

Thanks to \eqref{eq: inf_sub_lerror_1}, \eqref{eq: inf_sub_lerror_2}, \eqref{eq: inf_sub_lerror_3}, \eqref{eq: inf_sub_lerror_4} and \eqref{eq: inf_sub_lerror_5}, we finally obtain \eqref{eq: inf-sup bound_idea}. 
 \end{proof}

Lemma \ref{lem: kernel_b_h_cut} implies that for any $v_h\in \Ker b_h$,
\begin{equation*}
    \tilde{a}_h(v_h,v_h)= a_h(v_h,v_h)\gtrsim \nrmL[h]{v_h}^2= \nrmL[\D]{v_h}^2. 
\end{equation*}
Hence, we have the uniform coercivity of $\tilde{a}_h(\cdot,\cdot)$ on $\mathrm{Ker}\,b_h$. Furthermore, the second variational equation of the mixed formulation \eqref{eq: Mixed formulation_cut} yields that $\tilde{u}_h$ belongs to $\Ker b_h$, and the well-posedness of the primal formulation \eqref{eq: Cutfem_Formulation} ensures that $\tilde{u}_h=u_h$. Finally, the Babuska-Brezzi theorem yields the well-posedness of \eqref{eq: Mixed formulation_cut}, thanks to Theorem \ref{thrm: inf-sup_cond_cut} which ensures the existence and uniqueness of the multiplier $\theta_h$.

\subsection{Local computation of multiplier}
A crucial feature of the mixed method is that each multiplier  $\theta_{h,i}$  can be computed locally, as sum of local contributions $\theta_N^{i} $ defined on the patches $\omega_N^i=\omega_N\cap \Omega_h^i$ associated to the nodes: 
\begin{equation}\label{eq:theta_sum}
 \theta_{h,i}=\displaystyle\sum_{N\in \N^i}\theta_N^{i},\qquad  i=1,2.
\end{equation}

Next, we present the local computation on each subdomain $\Omega_h^i$ ($i=1,2$). 

For this purpose, let the residual $r_h(\cdot):=l_h(\cdot) - \tilde{a}_h(u_h,\cdot)$. For convenience of notation, we also introduce for any  $v_{h,i}\in \mathcal D_h ^i$:
\[
 r_h^{1}(v_{h,1}) = r_h((v_{h,1},0)), \quad  r_h^{2}(v_{h,2}) = r_h((0, v_{h,2})).
\]
From \eqref{eq: Mixed formulation_cut}, we immediately have that
\begin{equation}\label{eq:local_multiplier}
        b_h^{i}(\theta_{h,i},v_{h,i}) = r_{h}^{i}(v_{h,i}) \quad \forall v_{h,i} \in \mathcal{D}_{h}^{i}\quad  (i=1,2),
\end{equation}
whereas from Lemma \ref{lem: kernel_b_h_cut}, we obtain that $r_h^i(v_{h,i}) =0$ for any $v_{h,i}\in\mathcal{C}_h^i$ and $i=1,2$.

Let now $i\in\{1,2\}$ and $N\in \N^i$. We define $\theta_N^i\in \mathcal{M}_h^i$ living on $\mathcal{F}_N^i:=\mathcal{F}_N\cap \mathcal{F}_h^i$ such that, for any triangle $T\in \mathcal{\omega}_N^i$:
 \begin{eqnarray}
      b_h^{i}(\theta_N^i,\varphi_N\chi_T)&=&r_h^{i}(\varphi_N\chi_T),\label{theta1_6}\\
      b_h^{i}(\theta_N^i,\varphi_M\chi_T)&=&0, \qquad\qquad\quad \forall M\in \mathcal{N}_T\setminus \{N\}\label{theta2_6}
 \end{eqnarray}
where $\varphi_N,\varphi_M$ are the $P^1$-nodal basis functions associated to the nodes $N$ and $M$, respectively, $\chi_T$ is the characteristic function on $T$, and $\mathcal{N}_{T}$ is the set of nodes in $T$. Note that (\ref{theta2_6}) implies that $\theta_N^i(M)=0$ for all $M\in \mathcal{N}_T\setminus \{N\}$ and $T\in \mathcal{\omega}_N^i$. 

Using a similar technique as in \cite{Dana2016}, one can show that $\theta_N^i$ is well-defined, thanks to the constraint imposed in the space $\mathcal{M}_{h}^i$, and that $\tilde \theta_{h,i}:=\displaystyle\sum_{N\in \N^i}\theta_N^{i}$ satisfies the weak equation \eqref{eq:local_multiplier}. Hence, $\tilde \theta_{h}=(\tilde \theta_{h,1},\tilde \theta_{h,2})$ belongs to $\mathcal{M}_{h}$  and satisfies the mixed  problem (\ref{eq: Mixed formulation_cut}), so by uniqueness of its solution we get  $\tilde \theta_{h}= \theta_{h}$ and \eqref{eq:theta_sum} is checked.

Moreover, similarly to \cite{Aimene} we get that:
\begin{equation}\label{eQ:bound_theta_i}
 \bigg(\sum_{F\in \mathcal F_N\cap \mathcal F_h^i} h_F k_i^{2}\|\theta_N^i\|_F^2\bigg)^{1/2} \lesssim\displaystyle\sum_{T\in\omega_N^i}|r_h^i(\varphi_N\chi_T)|.
\end{equation}
In the sequel, we bound $\theta_N^i$ in terms of the solution $u_h$ and the data; this estimate is useful in the a posteriori error analysis. 

\begin{theorem}\label{thm: bound theta_i_Pi_T^1}
For $i\in \{1,2\}$ and $N\in\N^i$, we have that:
\begin{equation*}
  \begin{split}
  \bigg(\sum_{F\in \mathcal F_N \cap \F^i} h_F k_i^{2}\|\theta_N^i\|_F^2\bigg)&^{1/2}\lesssim \sum_{T\in\omega_N^i} \sqrt{h_{T}}k_i\bigg(\| \jump{\partial_n u_h^i}\|_{\partial T\setminus \partial \omega_N^i} +\| \jump{\partial_n u_h^i}\|_{\partial T\cap \Fg^i}\bigg) \\
        &+\sum_{T\in\T^\Gamma\cap \omega_N^i}\bigg(\frac{{k_\Gamma}}{\sqrt{ h_T}}\|[u_h]\|_{\Gamma_T}+\sqrt{h_{T}}\|g-[K\nabla u_h\cdot n_\Gamma]\|_{\Gamma_T}\bigg)\\ &+\sum_{T\in\omega_N^i} h_T\|f\|_{T^i}.
   \end{split} 
   \end{equation*}
\end{theorem}
\begin{proof}
Thanks to inequality \eqref{eQ:bound_theta_i}, we only need to bound the residual $r_h^i(\varphi_N\chi_T)$. 
 In the following, without loss of generality, we prove the bound for $i=1$. For any $T\in \omega_N^1$, we have that:
\begin{equation*}
    \begin{split}
        r_h^1(\varphi_N\chi_T) 
        &=\int_{T^1}f\varphi_N\dx
        + \int_{\Gamma_T}g\omega_2\varphi_N\ds
        - \int_{T^1}k_1\nabla u_{h,1} \cdot\nabla \varphi_N\dx\\
        &-\gamma h_T^{-1} \int_{\Gamma_T} k_{\Gamma}[u_h] \varphi_N\ds
        + \int_{\Gamma_T} \left( \{K\nabla u_h\cdot n_{\Gamma}\}\varphi_N  
        +\omega_1k_1\nabla\varphi_N\cdot n_{\Gamma} [u_h] \right)\ds\\
        & + \sum_{F \in \mathcal{F}_{T}\cap\F^1}\int_{F \cap \Omega^{1}}\langle k_1\nabla u_{h,1}\cdot n_{F}\rangle \jump{\varphi_N\chi_T}\ds\\
       &-\sum_{F\in \mathcal{F}_{T}\cap\mathcal F_g^1}\gamma_g h_F
        \int_{F}\jump{k_1\nabla u_{h,1}\cdot n_{F}} \jump{\nabla (\varphi_N\chi_T) \cdot n_{F} }\ds.
    \end{split}
\end{equation*}
Here above, we have used that 
 $[\varphi_{N} \chi_{T}]_{|\Gamma_{T}} = 
(\varphi_{N} )_{|\Gamma_{T}}$ since $n_{\Gamma}$ points from $\Omega^{1}$ to $\Omega^{2}$.
Using integration by parts for the third term on the right-hand side further yields: 
\begin{equation*}
     \begin{split}
        r_h^1(\varphi_N\chi_T)
        =&\int_{T^{1}}f\varphi_Ndx+\omega_2\int_{\Gamma_T}(g-[K\nabla u_h\cdot n_\Gamma])\varphi_Nds
        +\omega_1\int_{\Gamma_T}k_1\nabla\varphi_N\cdot n_\Gamma[u_h]ds \\
      &  - \gamma h_T^{-1}\int_{\Gamma_T} k_{\Gamma}[u_h]\varphi_N\ds
         - \frac{1}{2}\int_{\partial T^{1}\setminus \Gamma_{T}}k_1\jump{\partial_n u_{h,1}} \varphi_N\ds\\
          &-\sum_{F \in \mathcal{F}_{T}\cap\mathcal F_g^1 } \gamma_g h_F\int_{F}k_1\jump{\partial_n u_{h,1}} \jump{\nabla (\varphi_N \chi_{T}) \cdot n_{F}}\ds.
    \end{split}
\end{equation*}
If $T$ is not a cut element, then the integrals over $\Gamma_T$ vanish. By  the Cauchy-Schwarz inequality and using $ \omega_1 k_1=k_{\Gamma}$ and $\varphi_N=0$ on $\partial \omega_N^1$, we next get:
\begin{equation*}
    \begin{split}
        |r_h^1(\varphi_N\chi_T)|
        \lesssim &\, \|f\|_{T^{1}}\|\varphi_N\|_{T^{1}}
        +\omega_2\|g-[K\nabla u_h\cdot n_{\Gamma}]\|_{\Gamma_T}\|\varphi_N\|_{\Gamma_T} \\
        &+k_{\Gamma} \|[u_h]\|_{\Gamma_T} (\|\nabla\varphi_N\|_{\Gamma_T}
        +\gamma h_T^{-1}\|\varphi_N\|_{\Gamma_T})\\
         &+ \gamma_{g}h_F k_1\|\jump{\partial_n u_{h,1}}\|_{\partial T\cap \Fg^1}\|\nabla \varphi_N\|_{\partial T}
        + k_1\|\jump{\partial_n u_{h,1}}\|_{\partial T\setminus\partial\omega_N^1}\|\varphi_N\|_{\partial T}.
    \end{split}
\end{equation*}
Using the following bounds for the nodal basis function $\varphi_N$: 
\begin{equation*}
    \begin{split}
        &\nrmL[T^1]{\varphi_N}\leq \nrmL[T]{\varphi_N}\lesssim h_T,\quad  \nrmL[\partial T]{\varphi_N}\lesssim \sqrt{h_T},\\
      &  \nrmL[\partial T]{\nabla \varphi_N}\lesssim \frac{1}{\sqrt{h_T}},\quad \nrmL[\Gamma_T]{\varphi_N}\lesssim \frac{1}{\sqrt{h_T}}\nrmL[T]{\varphi_N}\lesssim \sqrt{h_T},
    \end{split}
\end{equation*}
one  finally gets, with $\omega_2\le 1$, that:  
\begin{equation}\label{eq: bound_r_h_cut}
    \begin{split}
      |r_h^1(\varphi_N\chi_T)|
        \lesssim &
        \,h_{T}\| f\|_{T^{1}} 
        +h_{T}^{1/2}\|g-[K\nabla u_h\cdot n_{\Gamma}]\|_{\Gamma_T} 
        + k_{\Gamma}h_T^{-1/2} \|[u_h]\|_{\Gamma_T}\\
          &+k_1h_T^{1/2}(\|\jump{\partial_n u_{h,1}}\|_{\partial T\cap \Fg^1}+\|\jump{\partial_n u_{h,1}}\|_{\partial T\setminus\partial\omega_N^1}).
    \end{split}
\end{equation} 
The terms on $\Gamma_T$ on the right-hand side of \eqref{eq: bound_r_h_cut} vanish when $T\in \T\backslash\T^\Gamma $. 
This ends the theorem's proof.
\end{proof}

\section{Local flux reconstruction}\label{sec:flux}
In this section, we propose a reconstruction of a discrete conservative flux $\sigma_h$, approximation of the continuous flux $\sigma:=K\nabla u$, based on the CutFEM solution $u_h$ and the multiplier $\theta_{h}$. An innovative feature is the use of an immersed Raviart-Thomas space (cf.  \cite{IRT}) on the cut elements, which leads to a reliable and locally efficient flux-based a posteriori error estimator.

In order to simplify the presentation, we assume in the sequel, without loss of generality, that no edge $F \in \F$ is situated entirely on $\Gamma$.

\subsection{The immersed Raviart-Thomas space $ \IRT^0(\mathcal T_h)$}
We begin by recalling the definition of the lowest-order immersed Raviart-Thomas space, recently introduced in \cite{IRT}. On a non-cut element, the polynomial space is the standard Raviart-Thomas space of lowest degree, that is
for any $ T \in\T\backslash \T^\Gamma$, we have: 
\begin{equation*}
    \RT^0(T)=\bigg\{\phi\in P^1(T)^2;\ \phi(x_1,x_2)= \left( \begin{array}{c}
         a  \\
         b 
    \end{array}\right) +c \left( \begin{array}{c}
         x_1  \\
         x_2 
    \end{array} \right),\ a,\,b,\,c \in \mathbb R\bigg\}.
\end{equation*}

In order to introduce the new finite element space on a cut cell $T\in\T^\Gamma$, let $t_\Gamma$ denote the unit tangent vector to $\Gamma$, oriented by a $90^\circ$ clockwise rotation of $n_\Gamma$, and recall that $T^i=T\cap \Omega^i$ ($i=1,2$). The local immersed Raviart-Thomas space $\IRT^0(T)$ is defined in \cite{IRT} as the set of  piecewise $\RT^0$- functions $\psi$, such that $\psi_i:=\psi_{|T^1}$ belong to $\RT^0(T)$ for $i=1,2$ and satisfy the following conditions:
\begin{equation}\label{ch4:eq:conditions_1}
\left\{
   \begin{array}{ll}
    &[\psi\cdot n_\Gamma]=\psi_1\cdot n_\Gamma-\psi_2\cdot n_\Gamma =0,\\
    &[K^{-1}\psi\cdot t_\Gamma](x_\Gamma)= k_1^{-1}\psi_1\cdot t_\Gamma(x_\Gamma)-k_2^{-1}\psi_2\cdot t_\Gamma(x_\Gamma)=0,\\
    &\div \psi_1= \div \psi_2.
   \end{array}
   \right.
\end{equation}
Here above, $x_\Gamma$ is an arbitrary point of $\Gamma_T$. We recall that for $i=1,2$, $(\psi_i\cdot n_\Gamma)_{|\Gamma_T}$ and $(\div \psi_i)_{|T}$ are constant, whereas $(\psi_i\cdot t_\Gamma)_{|\Gamma_T}$ is a priori linear. 

The condition $[\psi\cdot n_\Gamma]=0$ ensures that $\IRT^0(T)\subset H(\div,T)$ and that the (homogeneous) transmission condition across the interface is strongly satisfied. Meanwhile, the other condition on $\Gamma$, $[K^{-1}\psi \cdot t_\Gamma](x_\Gamma)=0$ takes into account the fact that $[\nabla u\cdot t_\Gamma](x_\Gamma)=0$, since $\gjump{u}=0$. Finally, the last condition of (\ref{ch4:eq:conditions_1}) ensures that $\text{dim}\, \IRT^0(T)=\text{dim} \, \RT^0(T) =3$.    

On each element $T\in \T$, the local degrees of freedom are the same as for the standard $\RT^0$ space, that is:   
\begin{equation}\label{eq: ddl_RT}
    N_{T,j}(\psi)=\frac{1}{|F_j|}\int_{F_j}\psi\cdot n_{T}\ds,\quad   1\le j\le 3,
\end{equation}
where $(F_j)_{1\le j\le 3}$ denote the edges of $T$. The global space $\IRT^0(\T)$ is then defined as the set of functions $\psi$ such that: for any $T\in \T\backslash \T^\Gamma $,  $\psi_{|T} \in \RT^0(T)$ whereas for any $T\in \T^\Gamma$,  $\psi_{|T} \in \IRT^0(T)$.
It is easy to check that $\psi$ satisfy the following  property:  
\begin{equation*}
\int_F\jump{\psi\cdot n_F}\ds:= \sum_{i=1}^{2}\int_{F^{i}}\jump{\psi_{i}\cdot n_F}\ds =0,\quad  \forall F\in \F^{int}.
\end{equation*}

Note that contrarily to the $\RT^0(\T)$ space, for a function $\psi\in \IRT^0(T)$ and a cut edge $F\in \F^\Gamma$, $(\psi\cdot n_F)_{|F}$ is only piecewise constant on the edge $F$. Thus, condition $\displaystyle \int_F\jump{\psi\cdot n_F}\ds=0$ does not imply $\jump{\psi\cdot n_F}_{|F}=0$. Hence, $\IRT^0(\T)\not \subset H(\div,\Omega)$.   

Next, we build a conservative flux in the $\IRT^0(\mathcal T_h)$ space. We consider the transmission condition  $[\sigma\cdot n_\Gamma]=g$ on $\Gamma$ in both the homogeneous and the non-homogeneous cases.

\subsection{Homogeneous Neumann transmission condition}\label{subsec:hom}
We assume here that $g=0$ and reconstruct a flux $\sigma_h$ in the space $\IRT^0(\mathcal T_h)$. 
This flux will then strongly satisfy the transmission condition across the interface, thanks to the definition of the immersed Raviart-Thomas space.  

We define $\sigma_h$ by imposing its degrees of freedom as follows:
 \begin{itemize}
 	\item for any $F\in \F^i\backslash \F^\Gamma\, (i=1,2)$, we set
  \begin{equation}\label{eq: def Flux_non cut }
     \int_F\sigma_h\cdot n_F\ds=  \int_F \langle k_i\nabla u_{h,i}\cdot n_F\rangle\ds -\int_F k_i\theta_{h,i}\ds.
 \end{equation}
 	\item  for any cut edge $F\in\F^\Gamma$, we set
 \begin{equation}\label{eq: def Flux_cut }
     \int_F\sigma_h\cdot n_F\ds= \sum_{i=1}^{2} \bigg(\int_{F^i}  \langle k_i\nabla u_{h,i} \cdot n_F\rangle\ds -\int_Fk_i\theta_{h,i}\ds\bigg ).
 \end{equation}
  \end{itemize} 
 
We can equivalently write the equations \eqref{eq: def Flux_non cut } and \eqref{eq: def Flux_cut } as follows:
\begin{equation*}
\begin{split}
\sigma_h\cdot n_F=&  \langle k_i\nabla u_{h,i}\cdot n_F\rangle - k_i\pi_F^0\theta_{h,i}, \quad \forall F\in \F^i\backslash\F^{\Gamma}\quad(i=1,2),\\
 \int_F\sigma_h\cdot n_F\ds= &\sum_{i=1}^{2} \bigg (\int_{F^i} \langle k_i\nabla u_{h,i}\cdot n_F\rangle\ds- k_i\pi_F^0\theta_{h,i}\bigg), \quad \forall F\in \F^\Gamma. 
 \end{split}
 \end{equation*}
Note that $\sigma_h\cdot n_F$ is only piecewise constant on the cut edges, but it belongs to $H(\div, T)$ for any cut triangle $T\in \T^\Gamma$. Next, we establish the conservation property.
\begin{theorem}\label{thrm: conservation_prop_iRT}
One has that
	\begin{equation}\label{eq: conservation_property_IRT}
		-(\div\, \sigma_h)_{|T}=\pi^0_T f,\quad \forall T\in\T.
	\end{equation}
\end{theorem}
\begin{proof} Let $T\in \T$. We start from  
$\displaystyle \int_{T} \div \sigma_h\dx= \int_{\partial T}\sigma_h\cdot n_T \ds$ and use the flux definition \eqref{eq: def Flux_non cut }-\eqref{eq: def Flux_cut }. On a non-cut cell, we obtain \eqref{eq: conservation_property_IRT} by testing the mixed formulation \eqref{eq: Mixed formulation_cut} with $(\chi_T,0)$ if $T\in \T^1$, and with $(0,\chi_T)$ if $T\in \T^2$. 

So in the sequel, we focus on a cut cell $T\in \mathcal T_h^{\Gamma}$ and test \eqref{eq: Mixed formulation_cut} with $v_h=(\chi_T,\chi_T)$. This yields that on any cell $T'\in \mathcal T_h$ one has that $(\nabla v_{h,i})_{|T'}=0$, hence 
\begin{equation}\label{eq:a_i}
a_i(u_{h,i},v_{h,i})=j_i(u_{h,i},v_{h,i})=0, \quad i=1,2.
\end{equation}
Moreover, one also has that
\begin{equation}\label{eq:a_Gamma}
a_{\Gamma}(u_h,v_h)= \sum_{T'\in\T^{\Gamma}}\int_{\Gamma_{T'}}\bigg( \frac{\gamma k_{\Gamma}}{h_{T'}}[u_h]-\{K\nabla u_h\cdot n_{\Gamma}\}\bigg)[v_h]\ds=0,
\end{equation}
since for any cut cell $T'\in\T^{\Gamma}$, one has $(v_{h,1})_{|T'}=(v_{h,2})_{|T'}$ and therefore, $[v_h]_{|\Gamma_{T'}}=0$.

Next, for any $F\subset \partial T$, we can write that 
$$\sigma_h\cdot n_T|_{F} =\sigma_h\cdot n_F \jump{v_h},$$
so we obtain, using \eqref{eq: def Flux_non cut } and \eqref{eq: def Flux_cut }, as well as $b_h(\theta_h,v_h) = \displaystyle \sum_{F \in \mathcal{F}_{T}} \int_F k_i\theta_{h,i}\ds$, that
\begin{equation*}
	-\int_{T}\div \sigma_h\dx=-\sum_{F\in \mathcal{F}_{T}} \int_{F}\sigma_h\cdot n_F\jump{v_h}\ds=-d_h(u_h,v_h)+b_h(\theta_h,v_h).
\end{equation*} 
Using next the definition of $\tilde a_h(\cdot,\cdot)$, as well as \eqref{eq:a_i} and \eqref{eq:a_Gamma}, we further get:
\begin{equation*}
	-\int_{T}\div \sigma_h\dx=\tilde a_h(u_h,v_h)+b_h(\theta_h,v_h)=l_h(v_h)=\int_{T}f\dx,
\end{equation*} 
which yields the desired relation \eqref{eq: conservation_property_IRT}.
\end{proof}

\subsection{Non-homogeneous Neumann transmission  condition}
 We can now treat the general case  $g\neq 0$.  For any $T\in \T^\Gamma$, we set $g_h=\pi_{\Gamma_T}^0g$ and define the linear continuous operator $\mathcal L_T:  \mathcal{RT}^0(T^1)\times \mathcal{RT}^0(T^2)\longrightarrow \mathbb R^6$ such that for any $\tau=(\tau_1,\tau_2)$,  
\begin{equation*}
	\mathcal L_T(\tau_1,\tau_2)	=\bigg(\bigg(\int_{F_j}\tau\cdot n_{F_j}\ds \bigg)_{1\leq j\leq3 } ,\,
		[\tau\cdot n_\Gamma],\,
		\div\tau_1 -\div\tau_2,\,
		[K^{-1}\tau\cdot t_\Gamma] (x_T)\bigg),
\end{equation*}
where on a cut side $ F$, we have   that 
$$ \displaystyle  \int_{F}\tau\cdot n_F\ds= \sum_{i=1}^{2}\int_{F^i}\tau_i\cdot n_F\ds.
$$

The operator $\mathcal L_T $ is injective due to the unisolvence of the $\IRT^0(T)$ space, and therefore surjective. Hence, there exists a unique flux 
$$\sigma^g_T=(\sigma^g_1,\sigma^g_2)\in  \mathcal{RT}^0(T^1)\times \mathcal{RT}^0(T^2)$$ such that $\mathcal L_T(\sigma^g_T)=(0,0,0,g_h,0,0)$. We denote by $ \sigma^g $ the zero extension to $\Omega$: 
$\sigma^g=\displaystyle\sum_{T\in \T^\Gamma} \sigma^g_T \chi_T$. We now define the global flux $\sigma_h^{g}$ as follows:  
\begin{equation}\label{eq:flux_g}
    \sigma_h^{g}= \sigma_h +\sigma^g,
\end{equation}
where $\sigma_h\in \IRT^0(\T)$ is the flux corresponding to $g=0$, defined by \eqref{eq: def Flux_non cut }-\eqref{eq: def Flux_cut }.

We can establish the following conservation property. On a cut triangle $T$, we use the discrete divergence operator $\div_h$ defined, for a function $\tau$ such that $\tau_{|T^i}\in H(\div,T^i)$,  by $(\div_h \tau)_{|T^i}=\div (\tau_{|T^i})$ for $i=1,2$.
\begin{theorem}
One has that
	\begin{equation}\label{eq: conservation_property_2}
		-(\div_h \sigma_h^g)_{|T}=\pi^0_T f,\quad \forall T\in\T.
	\end{equation}
\end{theorem}
\begin{proof}We treat here only the case of a cut cell $T\in\T^\Gamma$. As in the proof of Theorem \ref{thrm: conservation_prop_iRT}, we have: 
\begin{equation}\label{prop: sigma_1}
    \displaystyle	-\int_{T}\div\sigma_h\dx=l_h(v_h)=\int_{T}f\dx+ \int_{\Gamma_T} g\ds.
\end{equation}  
Note that we also have, integrating by parts on each $T^i$ and using the degrees of freedom of $\sigma^g$, that:
   \begin{equation}\label{prop: sigma_g}
   	\displaystyle \int_T\div_h\sigma^g\dx = \sum_{i=1}^{2}\displaystyle \int_{T^i}\div\sigma_i^g\dx=\int_{\Gamma_T}[\sigma^g\cdot n_\Gamma ]\ds=\int_{\Gamma_T}g_h\ds.
   \end{equation}
Using \eqref{prop: sigma_1}, \eqref{prop: sigma_g} and the definition \eqref{eq:flux_g}, we immediately obtain:
\[
\displaystyle \int_T\div_h\sigma^g_h\dx =\displaystyle \int_T\div \sigma_h\dx +\displaystyle \int_T\div_h\sigma^g\dx =-\int_{T}f\dx,
\]
and, hence, the announced result.       
\end{proof}

\section{Application to a posteriori error analysis}\label{sec: A posteriori_IRT}
For the sake of simplicity, we assume here that $g=0$.  

We set $\displaystyle \tau_h=K^{-1/2}(\sigma_h-K\nabla_h u_h)$, where $\sigma_h \in \IRT^0(\T)$ is the flux introduced in Subsection \ref{subsec:hom}, and we define the a posteriori local error estimator:
\begin{equation*}
    \eta_{T}
    =\|K^{-1/2}(\sigma_h - K\nabla_h u_h)\|_{T}=\|\tau_h \|_{T},\quad \forall T\in \mathcal T_h.
\end{equation*}
Since $u_h$ is discontinuous across $\Gamma$, we use the discrete gradient $\nabla_h$ in a cut triangle $T\in \T^\Gamma$; thus, $ \nabla_hu_h\in L^2(T)$ is defined by its $L^2$-restriction to each subdomain:
\[
(\nabla_h {u_{h}})_{|T\cap\Omega^{i}} = (\nabla {u_{h,i}})_{|T\cap \Omega^{i}},\quad 1\leq i\leq 2.
\]

In addition, on the cut cells $T\in \T^\Gamma$ we also consider
\begin{equation*}
    \tilde{\eta}_T=  \frac{\sqrt {h_T k_\Gamma}}{\sqrt {h_T^{min}|\Gamma_T|}}\| [u_h]\|_{\Gamma_T},\quad \forall T\in \T^\Gamma,
\end{equation*}
where  $h_T^{min}=\min\{|F^i|; \  F\in \partial T\cap \F^\Gamma, \  1\leq i \leq 2 \}$.

We introduce another local estimator on the cut edges: 
\begin{equation*}
	\eta_F=\frac{\sqrt{h_F}}{\sqrt{k_{\Gamma}}}\|\jump{\sigma_h\cdot n_F}-\pi_F^0\jump{\sigma_h\cdot n_F}\|_{F}=\frac{\sqrt{h_F}}{\sqrt{k_{\Gamma}}}\|\jump{\sigma_h\cdot n_F}\|_{F},\,\quad \forall F\in \mathcal F_h^{\Gamma}.
\end{equation*}

The corresponding global error estimators are given by: 
\begin{equation*}
  \eta =\bigg (\sum_{T\in \mathcal{T}_h}\eta_{T}^2\bigg)^{1/2},
        \qquad \eta_{\Gamma}=\bigg(
         \sum_{F\in \F^\Gamma}\eta_F^2+
         \sum_{T\in\T^\Gamma}\tilde{\eta}_T^2\bigg)^{1/2},
\end{equation*}
while the data approximation term is given by
\begin{equation*}
	\epsilon(\Omega)= 
         \left(\sum_{T\in \T}\frac{h_T^2}{\delta_T}\|f-\pi^0_T f\|_{T}^2 \right)^{1/2}, \qquad 
         \delta_T= \left\{
	\begin{array}{ll}
		k_i & \text{if}\ T\in \T^i\backslash\T^\Gamma,\\
		k_\Gamma &\text {if} \ T\in \T^\Gamma.
	\end{array}\right.
\end{equation*}

We have established in \cite{Article2} the following error bounds regarding the reliability and local efficiency of the a posteriori estimator $\eta+\eta_{\Gamma}$.

\begin{theorem}[Reliability]
Let $u$ and $u_h$ be the solutions of \eqref{eq: continuous_problem} and \eqref{eq: Cutfem_Formulation}, respectively. There exists a constant $C>0$ independent of the mesh, the coefficients and the interface such that   
	\begin{equation}\label{eq:reliab}
	\bigg(\sum_{i=1}^{2}\|k_i^{1/2}\nabla (u-u_{h,i})\|^2_{\Omega^i}\bigg)^{1/2}\leq \eta +C\left( \eta_{\Gamma}+\epsilon(\Omega)\right).
	\end{equation}
\end{theorem}

The previous Theorem shows that the $H^1$-seminorm of the error is bounded by the main estimator $\eta$ with a reliability constant equal to $1$, in agreement with well-known results for equilibrated flux-based estimators. The additional estimator $\eta_{\Gamma}$ and the higher-order term $\epsilon(\Omega)$ in estimate (\ref{eq:reliab}) are multiplied by a constant which is independent of the mesh size, the diffusion coefficients and the interface geometry.

The local efficiency is established with respect to the following norm of the error:
\begin{equation*}
		\| v_h\|_{h,\Delta_T}^2= \sum_{i=1}^{2}\bigg(\|k_i^{1/2}\nabla v_{h,i}\|^2_{\Delta_T\cap \Omega^i}+j_{i,\Delta_T}(v_{h,i},v_{h,i})\bigg)
	+\sum_{T\in \mathcal{T}_h^{\Gamma}\cap \Delta_T}\int_{\Gamma_T}\frac{k_{\Gamma}}{h_T}[v_h]^2 \ds,
\end{equation*}
where 
$$\displaystyle\Delta_T=\bigcup_{N\in \mathcal N_T}\omega_N,\quad j_{i,\Delta_T}(v_{h,i},v_{h,i}):=\sum_{F\in\mathcal{F}_g^i\cap\Delta_T} h_F \int_{F}k_i\jump{\partial_n v_{h,i}}^2\ds\quad (i=1,\,2).$$
We recall that $\mathcal N_T$ is the set of vertices of $T$ and $\omega_N$ the set of triangles sharing the node $N$. It is also useful to introduce, for $F\in \F^{\Gamma}$, the notation $\displaystyle\Delta_F=\bigcup_{T,\,\partial T\supset F}\Delta_T$.

In the following, we give the local bound for each estimator $\eta_T$, $\tilde \eta_T$ and $\eta_F$. The main difficulty lies in obtaining the robustness with respect to the coefficients and interface geometry on the cut elements $T\in \T^\Gamma$. For theoretical reasons only, to establish the efficiency of $\eta_T$ with the best constant, we make the following assumption.

\begin{assumption}\label{ass: 1}
For any $T\in \T^\Gamma$, there exist closed, regular shaped triangles $\tilde{T}^1\subset \Omega^1,\tilde{T}^2\subset \Omega^2$ such that  they have  $\Gamma_T$ as a common side:  $ \tilde{T}^1\cap \tilde{T}^2=\Gamma_T$.
\end{assumption}

\begin{theorem}\label{thrm: efficiency_IRT}
Under Assumption \ref{ass: 1}, for any $ T\in \T$ there exists a positive constant $C_T$ such that
	\begin{equation}\label{estim_eff}
		\forall T\in \T,\quad \eta_T \lesssim C_{T} \left(\|u-u_h\|_{h,\Delta_T}+\epsilon(\Delta_T)\right),
	\end{equation}
with $C_T=1$ if $T\in \T\backslash \T^\Gamma$ and $C_{T}= \displaystyle \max_{T'\in \Delta_T \cap \T^{\Gamma}}\frac{h_{T'}^{1/2}}{|\Gamma_{T'}|^{1/2}}\frac{k_{max}^{3/2}}{k_{min}^{3/2}}$ if $T\in \T^\Gamma$.
\end{theorem}

\begin{remark}
In the particular case where Assumption 1 might not hold, we can still prove cf. \cite{Article2} a similar estimate to (\ref{estim_eff}) but with an efficiency constant multiplied by $\sqrt{k_{max}}/\sqrt{k_{min}}$. It is worth noting that we have not noticed any influence of Assumption 1 in the numerical experiments, including the petal-shaped domain of \Cref{ex3} which presents a complex mesh/interface geometry.
\end{remark}
\begin{theorem}\label{thrm: bound_eta_effi} 
Let $ T\in \T^\Gamma $ and $F\in \F^\Gamma$. There exist positive constants $\tilde{C}_T$ and $C_F$ such that
	\begin{equation}
		\tilde \eta_T \lesssim \tilde{C}_T\|u-u_h\|_{h,T}, \qquad \eta_F \lesssim C_{F} \left(\|u-u_h\|_{h,\Delta_F}+\epsilon(\Delta_F)\right),
	\end{equation}
where $\displaystyle \tilde{C}_T= \frac{h_T}{\sqrt{h_T^{min}|\Gamma_T|}}$ and $\displaystyle C_F=\max_{T\in \Delta_F\cap \T^{\Gamma}}\frac{h_T}{|\Gamma_T|}\frac{k_{max}^{3/2}}{k_{min}^{3/2}}$.
\end{theorem}

Theorems \ref{thrm: efficiency_IRT} and \ref{thrm: bound_eta_effi} assert the local efficiency, with explicit bounds of the efficiency constants. On a cut cell $T$ and a cut edge $F$, $C_T$ and $C_F$ depend in theory on the ratio $k_{max}/k_{min}$; however, the numerical behavior of the global estimator appears to be quite robust with respect to this ratio, as shown in Figure \ref{fig:Conv_Ellipse} where we successfully tested a ratio of $10{,}000$. On the cut elements, the three efficiency constants also depend on the ratio $h_T/|\Gamma_T|$, which is $O(1)$ for most elements and again does not seem to influence the numerical tests.

\section{Numerical simulations}\label{sec:num_sim}

We present several numerical experiments to illustrate the theoretical results established in the previous sections. The numerical implementation is based on the open-source library FEniCS, along with the CutFEM library developed by Farina et al.~\cite{farina2021cut}, which is built based on FEniCS. Additional technical details regarding the challenges of implementing flux reconstruction on cut elements are provided in Appendix~\ref{appendix}.

For the stabilization parameters in the discrete problem, we set \(\gamma = 10\) and \(\gamma_g = 0.1\). The mesh refinement follows D\"orfler's marking strategy~\cite{dorfler1996convergent}, i.e. we look for the set of elements \(\mathcal T_h^m\) with minimal cardinal such that  \( \theta \eta(\mathcal T_h)^2 \leq \eta(\mathcal T_h^m)^2\). In the adaptive mesh refinement (AMR) procedure, the marking percent \(\theta\) is set to be 35\%, i.e. the ordered elements that account for the top 35\% of the total error estimator get refined. Although the reliability bound is established for the global error estimator $\eta+\eta_{\Gamma}$, we have observed in \cite{Article2} that the numerical results obtained when using $\eta+\eta_{\Gamma}$ (and the corresponding error indicator $\eta_T+\tilde \eta_T+\sum_{F\in \F^{\Gamma}\cap \partial T}\eta_F$ for any $T\in \T$) in the AMR procedure are very similar to those obtained with the estimator $\eta$ (and the indicator $\eta_T$) alone. The implementation of $\eta_\Gamma$ is more technical, and its use is also more expensive; therefore, in the following tests we employ \(\eta_T \) as error indicator in the AMR procedure, and $\eta$ as global error estimator.

In the following, we present three test cases.  All convergence curves are displayed on a log-log scale.

\begin{example}[Ellipse problem]\label{ex1}

Let $\Omega=[-1,1]^2$ and let $\Gamma$ be the ellipse centered at the origin of equation $\rho=1$, where $\rho ={\sqrt{\frac{x^2}{a^2}+\frac{y^2}{b^2}}}$ with $2a$ the width and $2b$ the height of the ellipse. Here, we take $a= \displaystyle \frac{\pi}{6.18}$ and $b=1.5a$. The exact solution of \eqref{eq: continuous_problem_weak} with $g=0$ is given by
\begin{equation*}
    u(x,y)=\left\{\begin{array}{ll}
         \dfrac{1}{k_1}\rho^p&  \text{if }\rho\leq 1 \\
        \dfrac{1}{k_2}\rho^p+\dfrac{1}{k_1}-\dfrac{1}{k_2} & \text{if }\rho>1
    \end{array}\right.,
\end{equation*}
where $p=5$. The diffusion coefficients in the two subdomains are $k_1=1$ (in the interior $\Omega^1$ of the ellipse) and $k_2=\mu k_1$, with $\mu>0$ a parameter that we let vary in the numerical experiments. 
\end{example}
 
We begin by testing the convergence rate of the flux reconstruction error between  the exact flux \(\sigma = K \nabla u\) and  the recovered flux \(\sigma_h\), that is \(\Vert K^{-1/2}(\sigma - \sigma_h) \Vert_{\Omega}\). In this test, we set \(\mu = 1\), so that \(k_1 = k_2\), and consider a smooth solution \(u \in H^2(\Omega)\). A uniform mesh refinement is applied. As shown in Figure~\ref{fig:FluxError_Ellipse}, we observe the expected optimal convergence rate \(O(h) = O(N^{-1/2})\) for both the (weighted) energy norm error \(\Vert K^{1/2} \nabla_h(u - u_h) \Vert_{\Omega}\) and the flux reconstruction error.

Figure~\ref{fig:Meshes_Ellipse} displays a sequence of adaptively refined meshes for \(\mu = 100\), from the initial mesh to the final one at iteration 15. The adaptive mesh refinement (AMR) procedure is terminated when the total number of degrees of freedom \(N\) reaches 30{,}000. In Figure~\ref{fig:Conv_Ellipse}, we report the convergence results for different values of \(\mu\), ranging from 10 to 10{,}000. In all cases, we observe the optimal convergence rate \(O(N^{-1/2})\) for both the weighted energy norm error and the global estimator $\eta$. These results confirm the robustness of the method with respect to the jump in the diffusion coefficients and validate the theoretical analysis.

 \begin{figure}[H]
    \centering
    \includegraphics[width=0.7\textwidth]{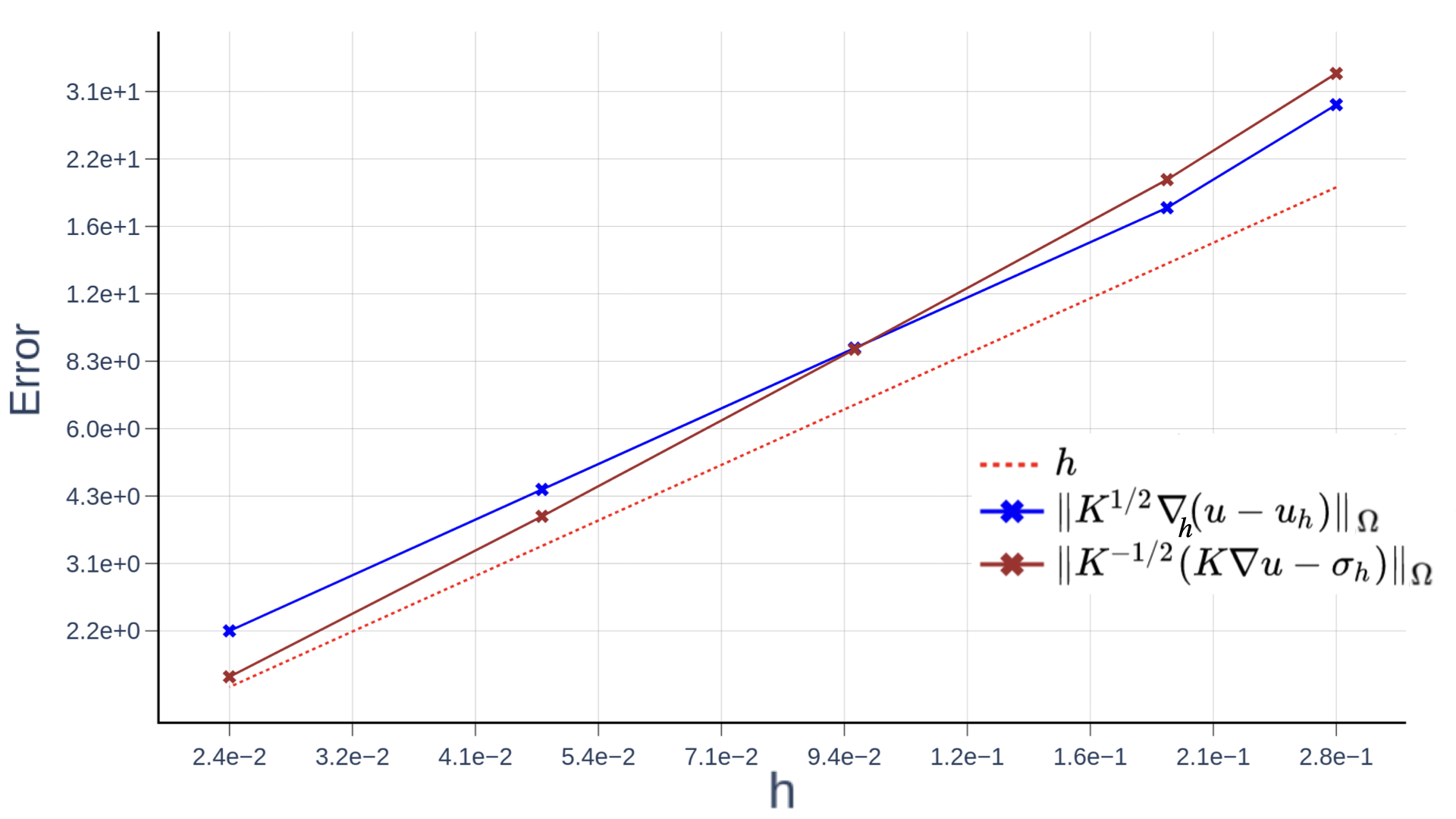}
    \caption{\Cref{ex1}. Convergence of errors for $\mu=1$ with uniform refinement}
    \label{fig:FluxError_Ellipse}
\end{figure}

\begin{figure}[ht!]
    \centering
    \begin{subfigure}[t]{0.45\textwidth}
        \includegraphics[width=1\textwidth]{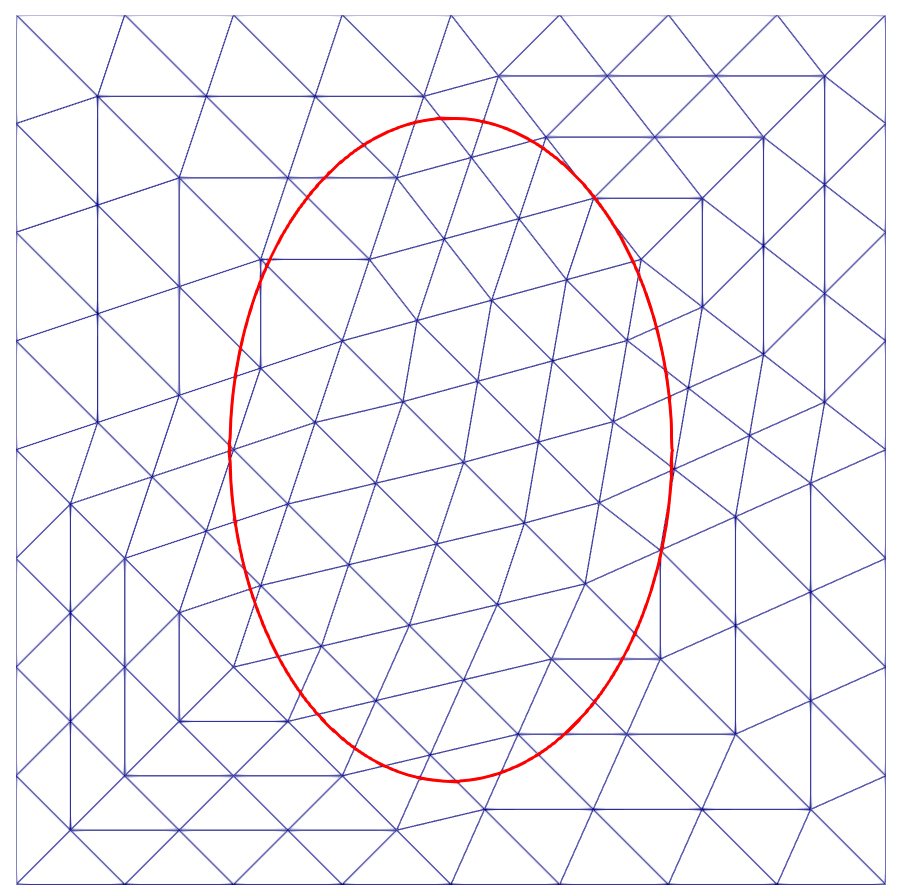}
        \caption{Initial mesh (iteration 0) }
    \end{subfigure}%
    \begin{subfigure}[t]{0.45\textwidth}
        \centering
        \includegraphics[width=1\textwidth]{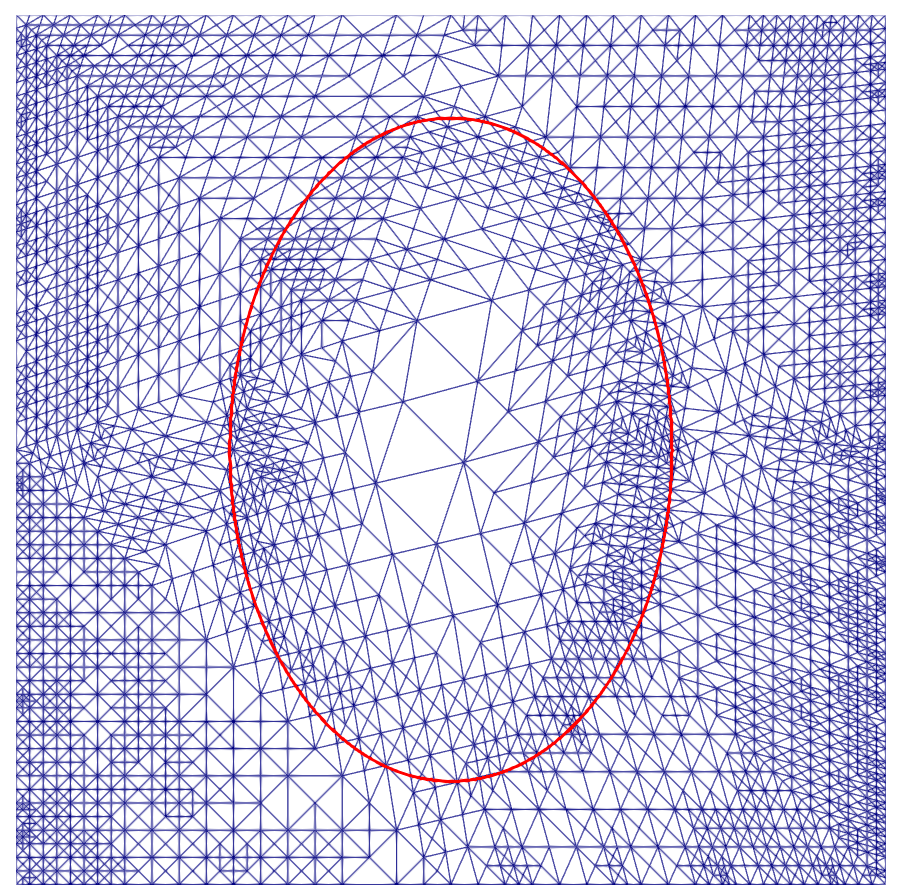}
        \caption{Iteration 9}
    \end{subfigure}
     \begin{subfigure}[t]{0.45\textwidth}
        \centering
        \includegraphics[width=1\textwidth]{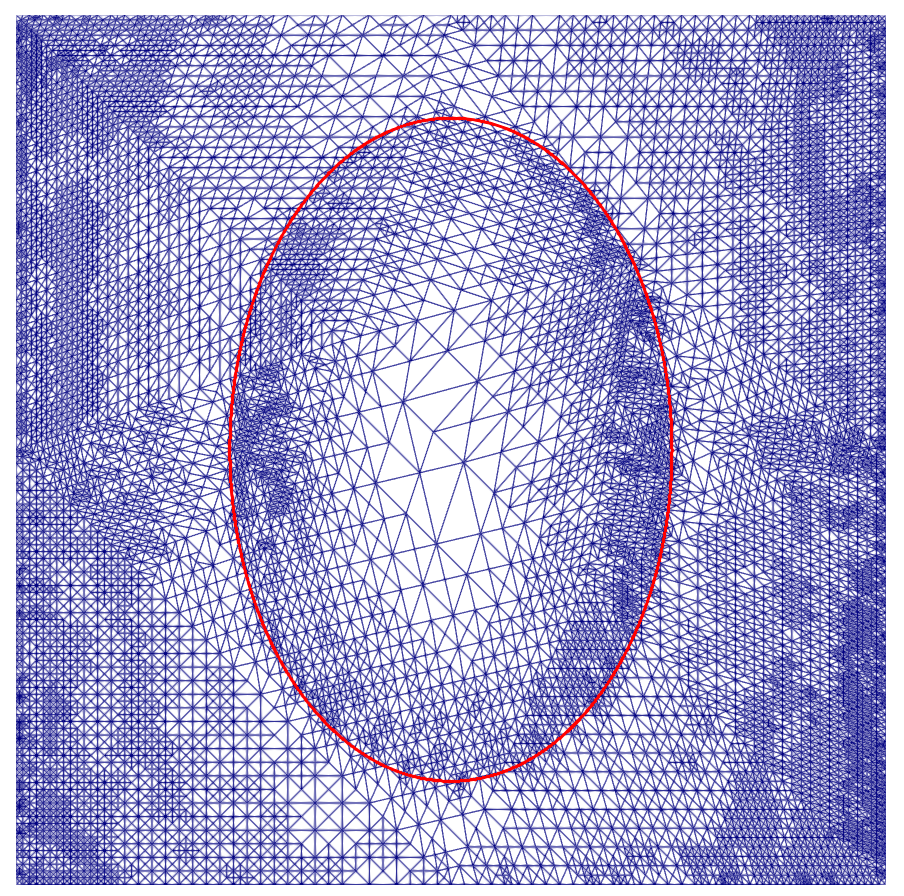}
        \caption{Iteration 13}
    \end{subfigure}
     \begin{subfigure}[t]{0.45\textwidth}
        \centering
        \includegraphics[width=1\textwidth]{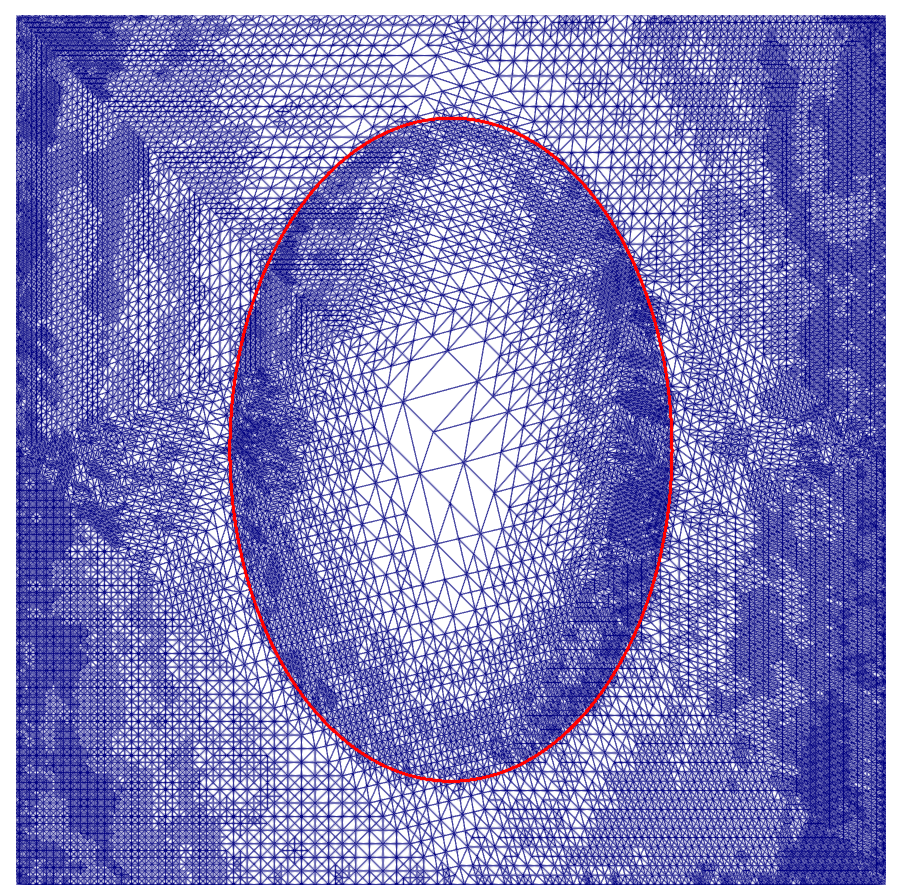}
        \caption{Iteration 15 (Final mesh)}
    \end{subfigure}
\caption{\Cref{ex1}. Sequence of adapted meshes for $\mu= 100$}
\label{fig:Meshes_Ellipse}
\end{figure}   

\begin{figure}[htp!]
 	\begin{subfigure}[t]{0.5\textwidth}
        \includegraphics[width=1\textwidth]{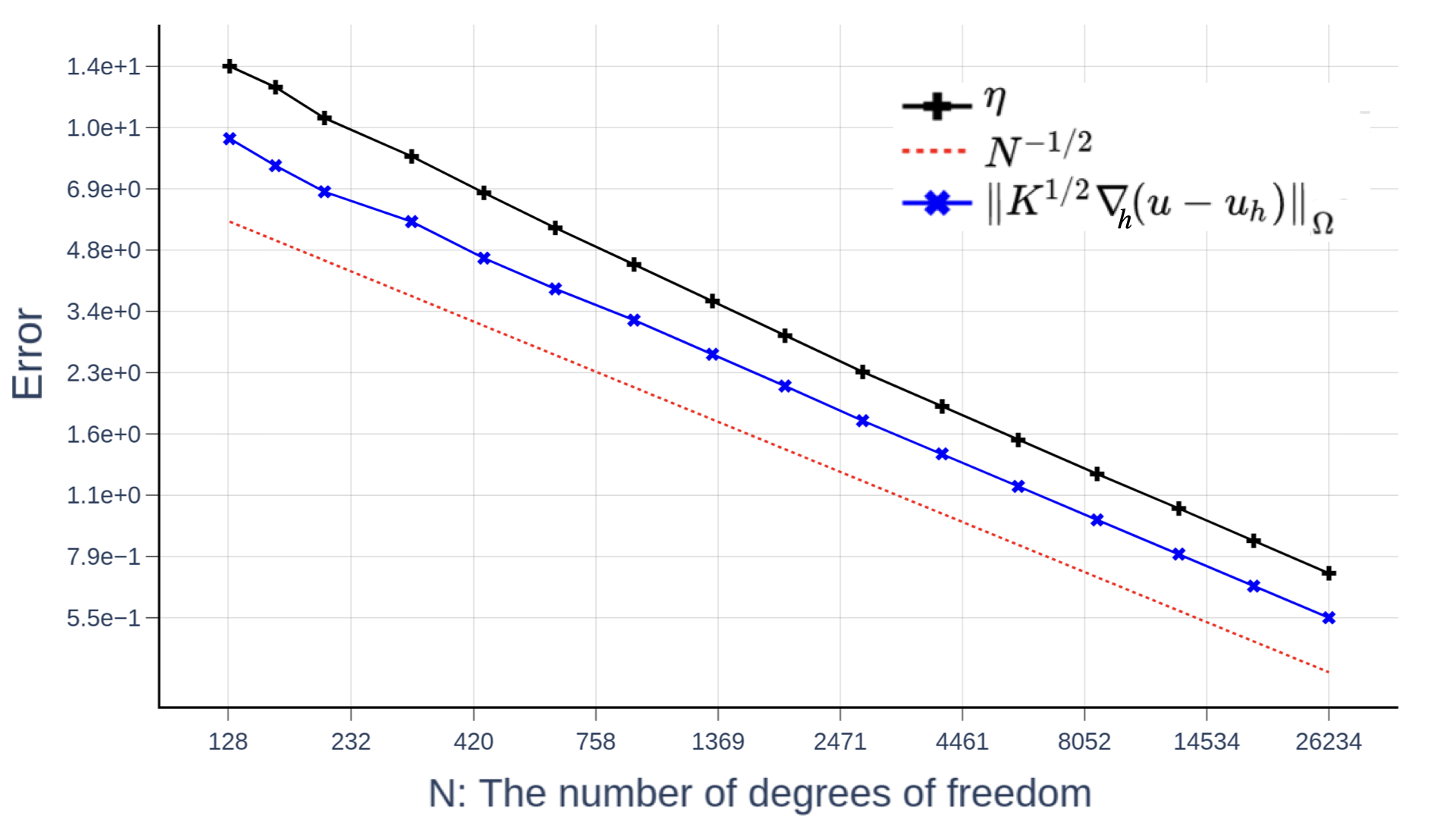}
        \caption{$\mu=10$}
         \end{subfigure}%
         \begin{subfigure}[t]{0.5\textwidth}
        \includegraphics[width=1\textwidth]{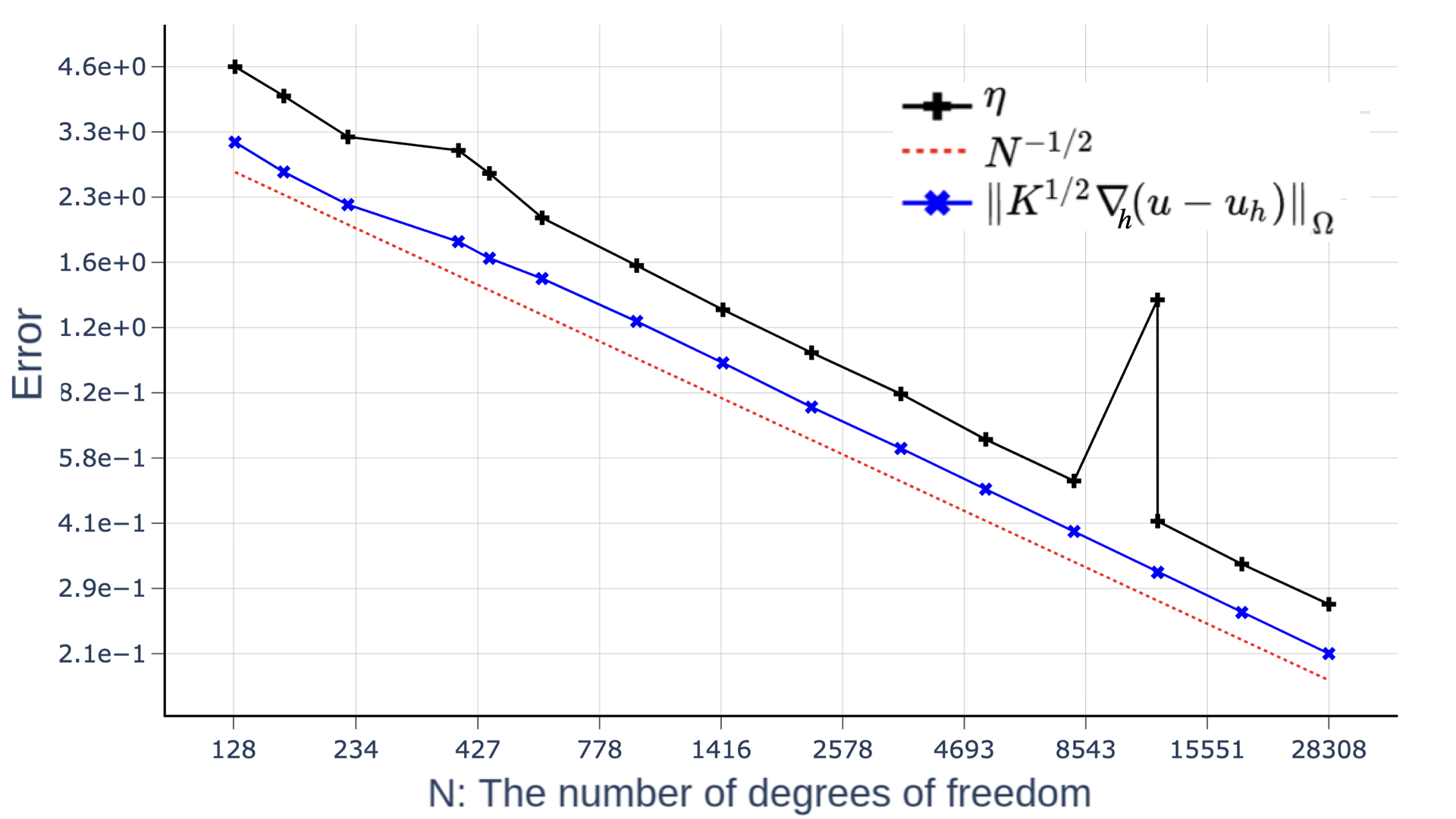}
        \caption{$\mu=100$}
         \end{subfigure}\\%
         \begin{subfigure}[t]{0.5\textwidth}
        \includegraphics[width=1\textwidth]{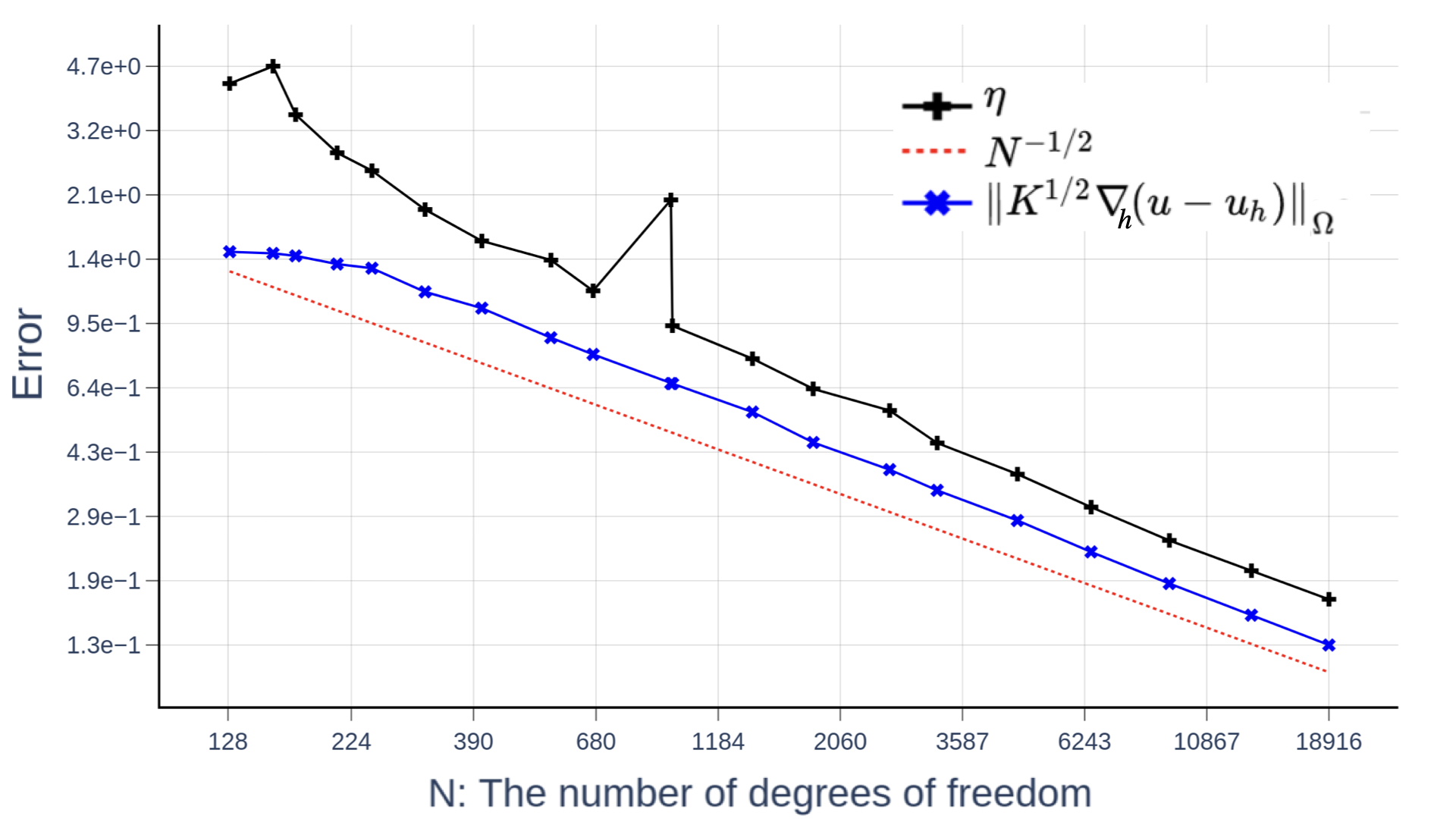}
        \caption{$\mu=1000$}
         \end{subfigure}%
         \begin{subfigure}[t]{0.5\textwidth}
        \includegraphics[width=1\textwidth]{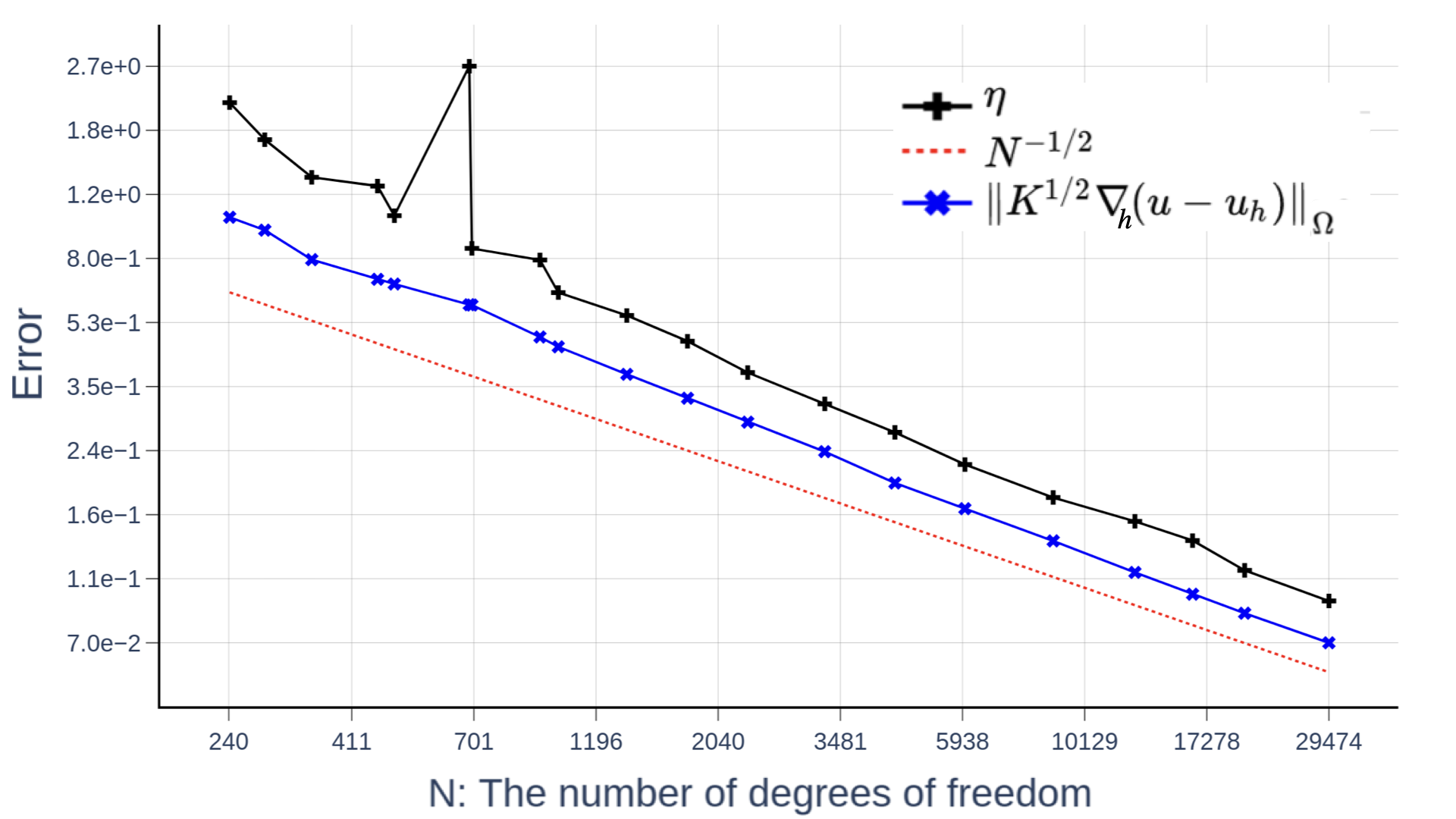}
        \caption{$\mu=10000$}
         \end{subfigure}%
    \caption{\Cref{ex1}. Convergence of the energy error and the error estimator for different $\mu$}
    \label{fig:Conv_Ellipse}
\end{figure}

\begin{example}[L--shaped problem]\label{ex2}
We now consider the L-shaped domain test case, see for instance \cite{Bonito_Devore_Nochetto}. The domain is $\Omega=[-5,5]\times[-5,5]\backslash [0,5]\times [-5,0]$ and presents again an interface, the circle centered at the origin and of radius $\rho_0=2\sqrt{2}$. The exact solution is given in polar coordinates $(\rho,\theta)$ by:
\begin{equation*}
    u(\rho,\theta)=\left\{\begin{array}{ll}
       \rho^{2/3}\sin({2\theta}/{3}),  & \text{if } \rho\leq \rho_0  \\
     \rho_{0}^{2/3}\sin({2\theta}/{3}) + \dfrac{2}{3\mu}\rho_{0}^{-1/3}\sin({2\theta}/{3})(\rho -\rho_0)  & \text{otherwise}
    \end{array}\right.,
\end{equation*}
whereas the diffusion coefficient is equal to $1$ inside the circle (in $\Omega^1$) and to $\mu$ outside the circle (in $\Omega^2$). We take here $\mu=5$.  
\end{example}

Figure~\ref{fig:Meshes_Lshape} shows a sequence of adaptively refined meshes. In this test, the AMR procedure is stopped when the total number of degrees of freedom \(N\) reaches 60{,}000. As expected, refinement occurs both near the interface and around the reentrant corner, where the solution exhibits a singularity. Figure~\ref{fig:Conv_Lshape} illustrates that both the weighted energy norm error and the a posteriori error estimator $\eta$ converge optimally at the rate \(O(N^{-1/2})\).

\begin{figure}[ht!]
    \centering
    \begin{subfigure}[t]{0.45\textwidth}
        \includegraphics[width=1\textwidth]{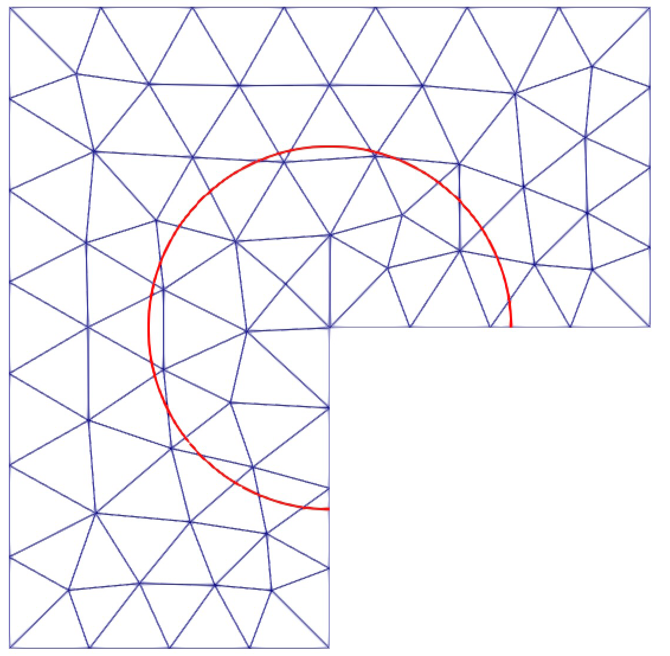}
        \caption{ Initial mesh (iteration 0)}
    \end{subfigure}
    \hspace{.3cm}
    \begin{subfigure}[t]{0.45\textwidth}
        \includegraphics[width=1\textwidth]{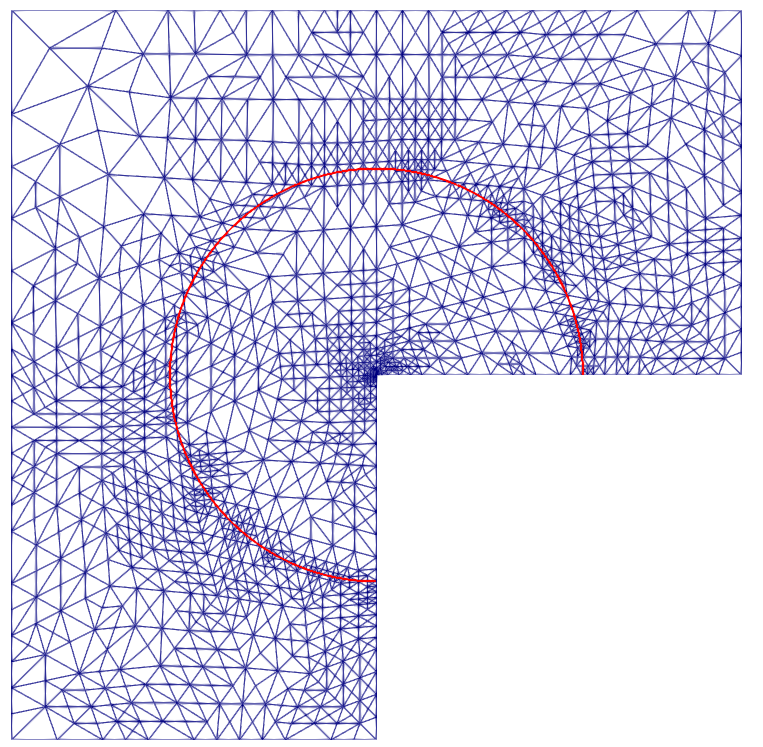}
        \caption{Iteration 10}
    \end{subfigure}\\
    \begin{subfigure}[t]{0.45\textwidth}
        \includegraphics[width=1\textwidth]{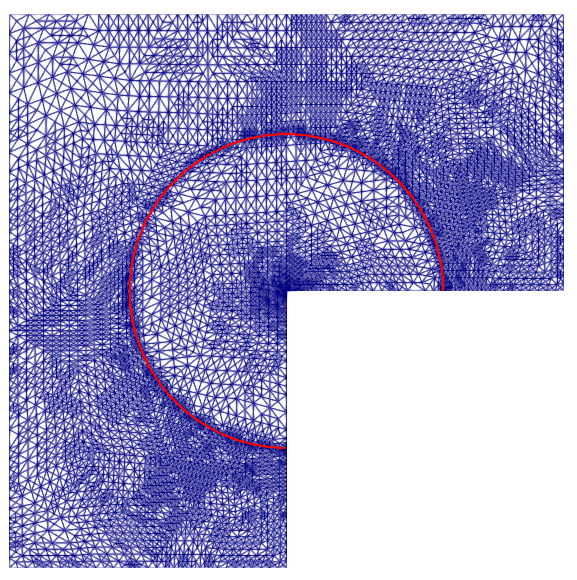}
        \caption{Iteration 15}
    \end{subfigure}
    \hspace{.3cm}
    \begin{subfigure}[t]{0.45\textwidth}
        \centering
        \includegraphics[width=1\textwidth]{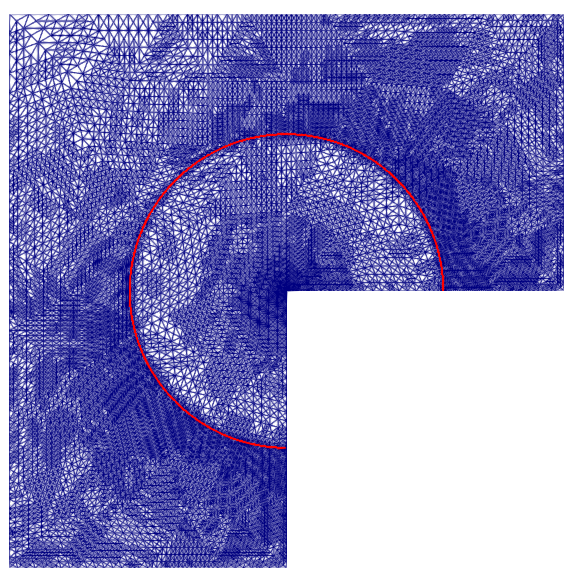}
        \caption{Iteration 18 (final mesh)}
    \end{subfigure}
    \caption{\Cref{ex2}. Sequence of adapted meshes}
    \label{fig:Meshes_Lshape}
\end{figure}
\begin{figure}[ht!]
\centering
        \includegraphics[width=0.7\textwidth]{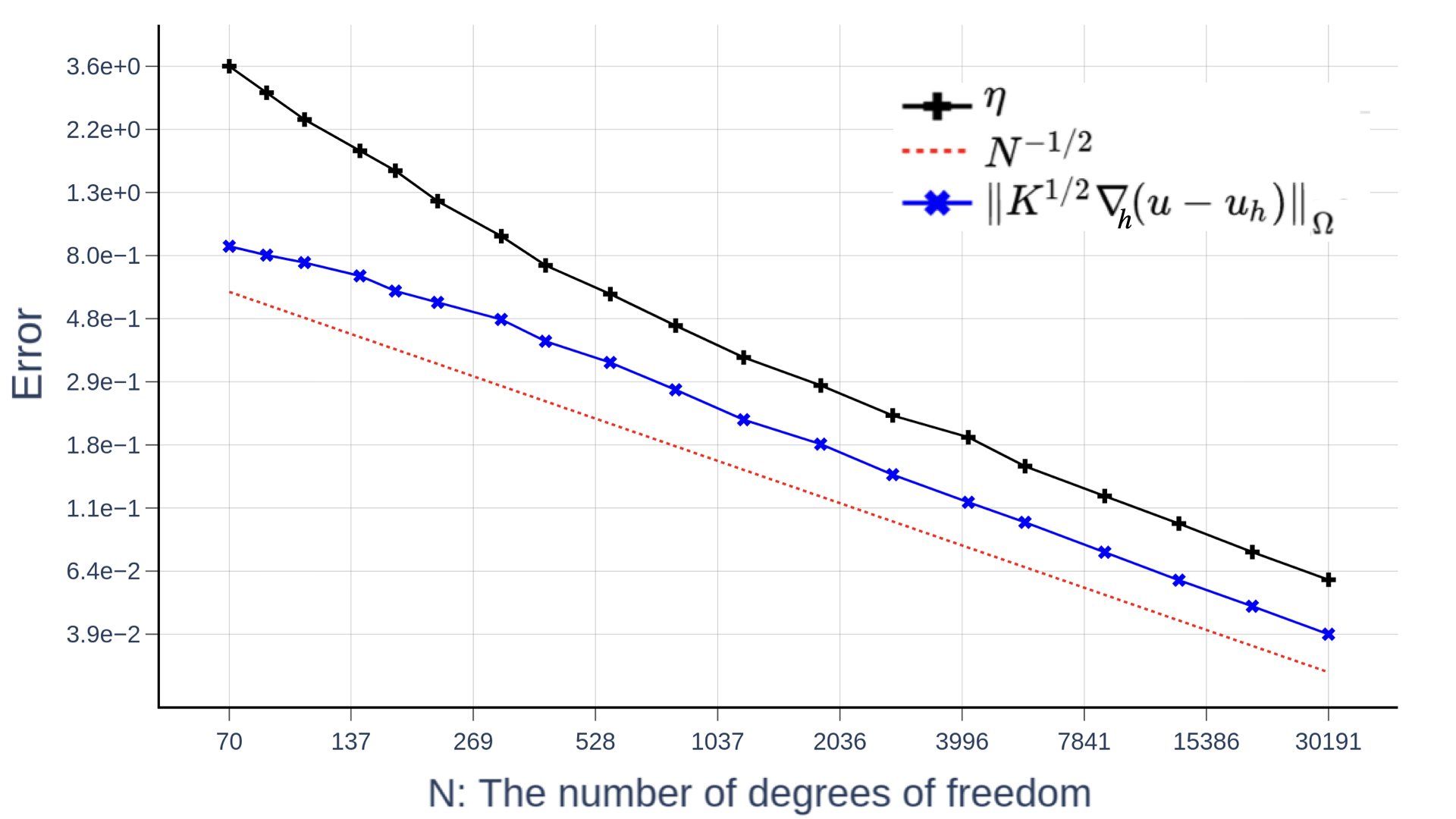}
\caption{\Cref{ex2}. Convergence of the energy error and the error estimator}
\label{fig:Conv_Lshape}
\end{figure}

\begin{example}[Petal-shaped  problem] \label{ex3}
Finally, we consider an interface problem characterized by a complex interface shape. The exact solution is described by a petal-shaped interface and is defined using the following level set function:
\[
u(x, y) = 
\begin{cases} 
 \phi(x, y), & \text{if } \phi(x, y) < 0 \\
\dfrac{1}{\mu} \phi(x, y), & \text{if } \phi(x, y) \geq 0,
\end{cases}
\qquad \forall (x, y) \in \Omega = [-1, 1]^2
\]
with $\mu=100$. Here, the level set function \(\phi(x, y)\) is given by:
\[
\phi(x, y) = (x^2 + y^2)^2 \left(1 + 0.5 \sin\left(12 \tan^{-1}\left(\frac{y}{x}\right)\right)\right) - 0.3.
\]
\end{example}
The AMR stopping criterion we set for this example is that the total number of degrees of freedom \(N\) remains below 20{,}000.  Figure~\ref{fig:Meshes_Petal} shows a sequence of adaptively refined meshes, where significant refinement is observed around the interface. This behavior is likely due to the higher curvature of the interface in this example.

\begin{figure}[ht!]
    \centering
    \begin{subfigure}[t]{0.45\textwidth}
        \includegraphics[width=1\textwidth]{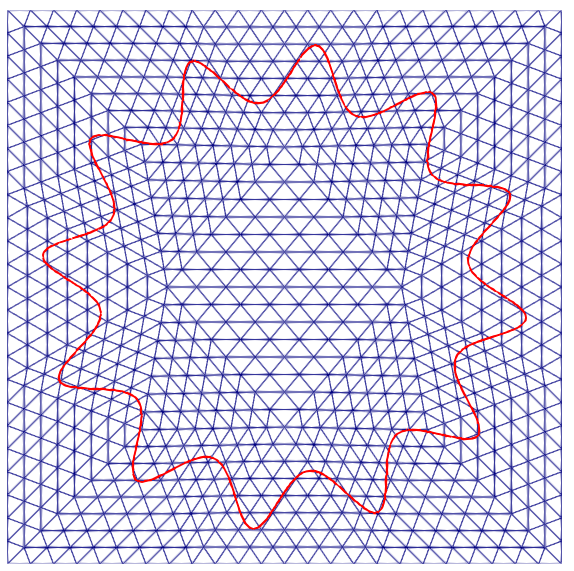}
        \caption{Initial mesh (iteration 0)}
    \end{subfigure}%
    \begin{subfigure}[t]{0.45\textwidth}
        \centering
        \includegraphics[width=1\textwidth]{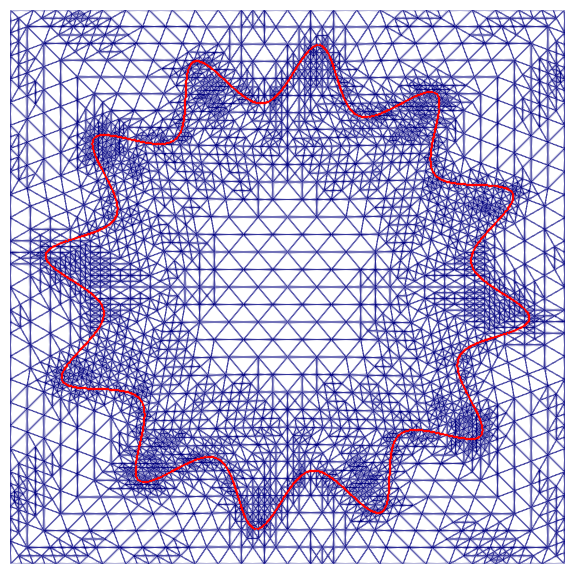}
        \caption{Iteration 5}
    \end{subfigure}\\ 
     \begin{subfigure}[t]{0.45\textwidth}
        \centering
        \includegraphics[width=1\textwidth]{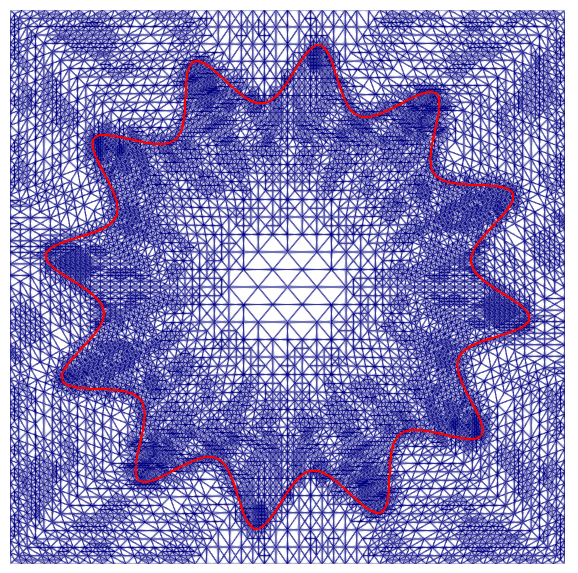}
        \caption{Iteration 8}
         \end{subfigure}
     \begin{subfigure}[t]{0.45\textwidth}
        \centering
        \includegraphics[width=1\textwidth]{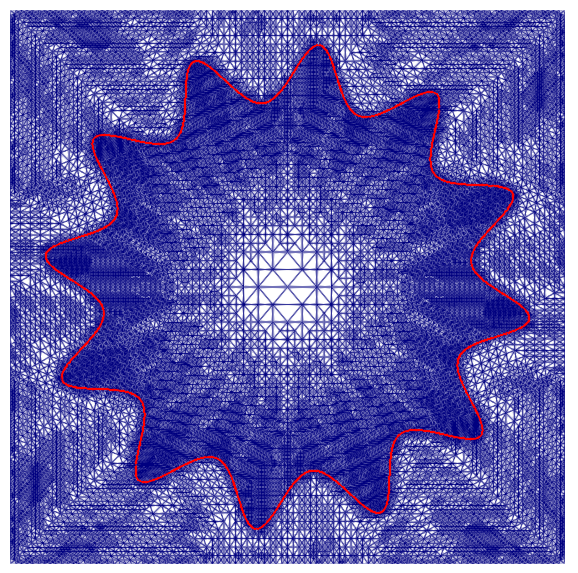}
        \caption{Iteration 10 (final mesh)}
    \end{subfigure}
    \caption{\Cref{ex3}. Sequence of adapted meshes}
    \label{fig:Meshes_Petal}
\end{figure}

The convergence plot of Figure \ref{fig:Conv_Petal} indicates the optimal rate decay $O(N^{-1/2})$ for both the error and the a posteriori error estimator $\eta$.
\begin{figure}[ht!]
\centering
    \includegraphics[width=0.7\textwidth]{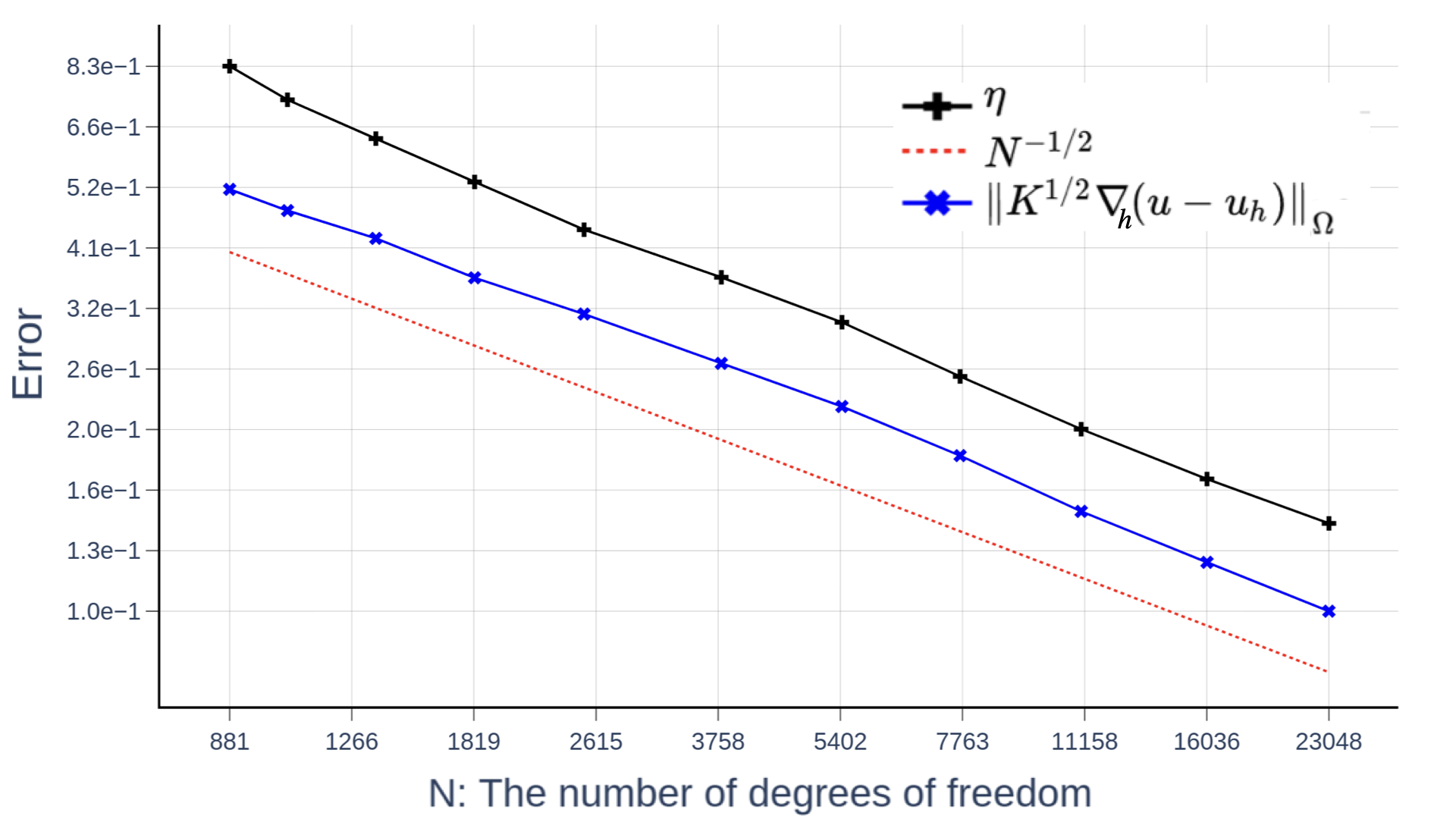}      
    \caption{\Cref{ex3}. Convergence of the energy error and the error estimator}
    \label{fig:Conv_Petal}
\end{figure}

\appendix
\section{Numerical implementation of the flux on cut elements}\label{appendix}
Let $ T\in \T^\Gamma$ and $\sigma_h \in \mathcal{IRT}^0(T)$. There exists a couple $(\sigma_{h,1},\sigma_{h,2}) \in \mathcal{RT}^0(T)\times \mathcal{RT}^0(T)$ such that 
\begin{equation} \label{first_condition}
	[\sigma_h\cdot n_\Gamma]={\sigma_{h,1}\cdot n_\Gamma} - {\sigma_{h,2}\cdot n_\Gamma}=0,
\end{equation}
\begin{equation} \label{second_condition}
	[K^{-1}\sigma_h\cdot t_\Gamma](x_\Gamma)={k_1^{-1}\sigma_{h,1}(x_\Gamma)\cdot t_\Gamma}_{|\Gamma_T} - {k_2^{-1}\sigma_{h,2}(x_\Gamma)\cdot t_\Gamma}_{|\Gamma_T}=0, 
\end{equation}
\begin{equation}\label{last_condition}
	\mathrm div\, \sigma_{h,1}= \mathrm div \, \sigma_{h,2}.
\end{equation}

The idea for the implementation of the immersed Raviart-Thomas space is to express the degrees of freedom of $\sigma_{h,1}$ and $\sigma_{h,2} $ in terms of those of $\sigma_h$. Thus, instead of implementing the flux $\sigma_h \in \IRT^0(\T)$, which is challenging due to its discontinuity across cut edges and the available data types in FEniCS, we have chosen to implement the two functions $ \sigma_{h,1} $ and $ \sigma_{h,2} $, which belong to the standard Raviart-Thomas space. 

We recall that the local degrees of freedom of $\mathcal{RT}^0(T)$ are given in \eqref{eq: ddl_RT} and that the basis function $\Psi_j$ associated to the edge $F_j\in \partial T$ is given by 
\begin{equation*}
	\Lambda_{T,j}(x)=\frac{|F_j|}{2|T|}\overset{\longrightarrow}{A_jx},\quad 1\le j\le 3,
\end{equation*}
where $A_j$ is the vertex of $T$ opposite to $F_j$. One further has the unique decomposition:
\begin{equation*}
	\sigma_{h,i}=\sum_{j=1}^{3} N_{T,j}(\sigma_{h,i})\Lambda_{T,j}\quad (1\le i\le 2).
\end{equation*}

Since $\sigma_{h,i}\cdot n_\Gamma \in P^0(\Gamma_T)$ for $i\in \{1,2\}$, condition  \eqref{first_condition} can be written as  
\begin{equation} \label{first_condition_2}
\begin{split}
\sigma_{h,1}(x_\Gamma)\cdot n_\Gamma - \sigma_{h,2}(x_\Gamma)\cdot n_\Gamma=0 \Longleftrightarrow&\\
\sum_{j=1}^{3}|F_j|\overset{\longrightarrow}{A_jx_\Gamma}\cdot n_\Gamma (N_{T,j}(\sigma_{h,1})- N_{T,j}(\sigma_{h,2}))=0,&
\end{split}
\end{equation}
while \eqref{second_condition} yields
\begin{equation}\label{second_condition_2}
	\sum_{j=1}^{3}|F_j|\overset{\longrightarrow}{A_jx_\Gamma}\cdot t_\Gamma (k_1^{-1}N_{T,j}(\sigma_{h,1})-k_2^{-1} N_{T,j}(\sigma_{h,2}))=0.
\end{equation}
Since $\mathrm div \sigma_{h,i}\in P^0(T)$ for $i\in \{1,2\}$, condition \eqref{last_condition} is equivalent to $\displaystyle \int_T\mathrm div ( \sigma_{h,1}- \sigma_{h,2})\ dx=0$. Integration by parts further yields
\begin{equation}\label{last_condition_2}
	\sum_{j=1}^{3} |F_j|(N_{T,j}(\sigma_{h,1})- N_{T,j}(\sigma_{h,2}))=0.
\end{equation}

Denoting, for $1\le i\le 2$ and $1\le j\le 3$, the unknowns by $x_j^i:=|F_j|N_{T,j}(\sigma_{h,i})$ and the coefficients by $\alpha_j:=\overset{\longrightarrow}{A_jx_\Gamma}\cdot n_\Gamma$ and $\beta_j:=\overset{\longrightarrow}{A_jx_\Gamma}\cdot t_\Gamma$, conditions \eqref{first_condition_2}, \eqref{second_condition_2} and \eqref{last_condition_2} translate into the following linear system:
\begin{equation}\label{first system}
\left\{ \begin{array}{rl}
& \alpha_1 x_1^1+ \alpha_2 x_2^1+ \alpha_3 x_3^1 + \alpha_1 x_1^2+ \alpha_2 x_2^2+ \alpha_3 x_3^2=0\\
& k_1^{-1}\beta_1 x_1^1+ k_1^{-1}\beta_2 x_2^1+ k_1^{-1}\beta_3 x_3^1 + k_2^{-1}\beta_1 x_1^2+ k_2^{-1}\beta_2 x_2^2+ k_2^{-1}\beta_3 x_3^2=0\\
& x_1^1+x_2^1+x_3^1 -x_1^2-x_2^2-x_3^2=0.
\end{array}\right.
\end{equation}

Assume now, without loss of generality, that the non-cut edge of $T$ is $F_1$. Then one has that
\begin{equation} \label{equation 4}
N_{T,1}(\sigma_h)= \left\{
\begin{array}{ll}
		N_{T,1}(\sigma_{h,1}) \ &\text{if} \ F_1\subset \Omega^1\\
		N_{T,1}(\sigma_{h,2}) \ &\text{if} \ F_1\subset \Omega^2.
\end{array}
\right.
\end{equation}
Furthermore, for any $j\in\{2,3\}$, one has 
\begin{equation}\label{equation 5 and 6}
|F_j|N_{T,j}(\sigma_h) = \int_{F_j^1} \sigma_{h,1} \cdot n_T\ ds+ \int_{F_j^2} \sigma_{h,2} \cdot n_T\ ds
= \frac{|F_j^1|}{|F_j|} x_j^1+ \frac{|F_j^2|}{|F_j|} x_j^2.
\end{equation}
Assuming that $F_1\subset \Omega^1$ and denoting the coefficients $\omega_j^i:=\dfrac{|F_j^i|}{|F_j|}$, for $2\leq j\leq3$ and $1\leq i\leq2$, equations \eqref{equation 4} and \eqref{equation 5 and 6} can be equivalently written as follows:
\begin{equation}\label{second system}
\left\{ \begin{array}{rl}
		&x_1^1= |F_1|N_{T,1}(\sigma_h)\\
		&\omega_2^1x_2^1+ \omega_2^2x_2^2= |F_2|N_{T,2}(\sigma_h) \\
		& \omega_3^1x_3^1+ \omega_3^2x_3^2= |F_3|N_{T,3}(\sigma_h).
\end{array}\right.
\end{equation}

Finally, gathering together \eqref{first system} and \eqref{second system}, we obtain the linear system:
\begin{equation}\label{final_system}
\begin{pmatrix}
		1&0  &0  &0  &0  &0  \\
		0&\omega_2^1  &0  &0  &\omega_2^2  &0  \\
		0&0  & \omega_3^1  &0  &0  &\omega_3^2  \\
		1&1  &1  &-1  &-1  &-1  \\
		\alpha_1&\alpha_2  &\alpha_3  &-\alpha_1  &-\alpha_2  &-\alpha_3  \\
		k_1^{-1}\beta_1&k_1^{-1}\beta_2  &k_1^{-1}\beta_3  &-k_2^{-1}\beta_1  &-k_2^{-1}\beta_2  &-k_2^{-1}\beta_3 
	\end{pmatrix} 
	\begin{pmatrix}
		x_1^1 \\
		x_2^1 \\
		x_3^1\\
		x_1^2\\
		x_2^2 \\
		x_3^2
	\end{pmatrix}
	=\begin{pmatrix}
		b_1\\
		b_2\\
		b_3\\
		0\\
		0\\
		0
	\end{pmatrix}
\end{equation}
where the right-hand side term is known, thanks to the definition \eqref{eq: def Flux_non cut }-\eqref{eq: def Flux_cut } of the flux: $b_j= |F_j|N_{T,j}(\sigma_h)$ for $1\leq j\leq 3$. Solving (\ref{final_system}) allows to compute $N_{T,j}(\sigma_{h,i})$ for $1\leq j\leq 3$ and $1\le i\le 2$, and hence substitute $\sigma_h$ by two Raviart-Thomas functions.

\vsf
\hspace{-.6cm}\textit{Acknowledgments. This project has received funding from the European Union’s Horizon H2020 Research and Innovation under Marie Curie Grant Agreement N° 945416.}

\bibliographystyle{siam}
\bibliography{bibliographie}

\begin{thebibliography}{10}

\bibitem{Ai:07b}
{\sc M.~Ainsworth}, {\em A posteriori error estimation for discontinuous
  {G}alerkin finite element approximation}, SIAM Journal on Numerical Analysis,
  45 (2007), pp.~1777--1798.

\bibitem{Dana2016}
{\sc R.~Becker, D.~Capatina, and R.~Luce}, {\em Local flux reconstructions for
  standard finite element methods on triangular meshes}, SIAM J. Numer. Anal.,
  54 (2016), pp.~2684--2706.

\bibitem{Bonito_Devore_Nochetto}
{\sc A.~Bonito, R.~A. Devore, and R.~H. Nochetto}, {\em Adaptive finite element
  methods for elliptic problems with discontinuous coefficients}, SIAM J.
  Numer. Anal., 51 (2013), pp.~3106--3134.

\bibitem{BFH:14}
{\sc D.~Braess, T.~Fraunholz, and R.~H. Hoppe}, {\em An equilibrated a
  posteriori error estimator for the interior penalty discontinuous {G}alerkin
  method}, SIAM Journal on Numerical Analysis, 52 (2014), pp.~2121--2136.

\bibitem{braess2008equilibrated}
{\sc D.~Braess and J.~Sch{\"o}berl}, {\em Equilibrated residual error estimator
  for edge elements}, Math. Comput., 77 (2008), pp.~651--672.

\bibitem{brezzi2012mixed}
{\sc F.~Brezzi and M.~Fortin}, {\em Mixed and {H}ybrid {F}inite {E}lement (M},
  New York: Springer-Verlag.

\bibitem{CutFEM}
{\sc E.~Burman, S.~Claus, P.~Hansbo, M.~G. Larson, and A.~Massing}, {\em
  Cut{F}{E}{M}: discretizing geometry and partial differential equations}, Int.
  J. Numer. Meth. Eng., 104 (2015), pp.~472--501.

\bibitem{CutFEM2}
{\sc E.~Burman and P.~Hansbo}, {\em Fictitious domain finite element methods
  using cut elements: {I}{I}. {A} stabilized {N}itsche method}, Appl. Numer.
  Math., 62 (2012), pp.~328--341.

\bibitem{CaCaZh:20}
{\sc D.~Cai, Z.~Cai, and S.~Zhang}, {\em Robust equilibrated a posteriori error
  estimator for higher order finite element approximations to diffusion
  problems}, Numer. Math., 144 (2020), pp.~1--21.

\bibitem{cai2021generalized}
{\sc Z.~Cai, C.~He, and S.~Zhang}, {\em Generalized {P}rager--{S}ynge identity
  and robust equilibrated error estimators for discontinuous elements}, J.
  Comput. Appl. Math., 398 (2021), p.~113673.

\bibitem{Article2}
{\sc D.~Capatina and A.~Gouasmi}, {\em Elliptic interface problem approximated
  by {C}ut{F}{E}{M}: {II}. {A} posteriori error analysis based on equilibrated
  fluxes}, https://arxiv.org/abs/2507.06740,  (2025, submitted).

\bibitem{Aimene}
{\sc D.~Capatina, A.~Gouasmi, and C.~He}, {\em Robust flux reconstruction and a
  posteriori error analysis for an elliptic problem with discontinuous
  coefficients}, J. Sci. Comput., 98 (2024), p.~28.

\bibitem{capatina_he2021flux}
{\sc D.~Capatina and C.~He}, {\em Flux recovery for {C}ut {F}inite {E}lement
  {M}ethod and its application in a posteriori error estimation}, ESAIM: Math.
  Model. Numer. Anal., 55 (2021), pp.~2759 -- 2784.

\bibitem{dorfler1996convergent}
{\sc W.~D{\"o}rfler}, {\em A convergent adaptive algorithm for {P}oisson's
  equation}, SIAM J. Numer. Anal., 33 (1996), pp.~1106--1124.

\bibitem{ern2009accurate}
{\sc A.~Ern, I.~Mozolevski, and L.~Schuh}, {\em Accurate velocity
  reconstruction for discontinuous {G}alerkin approximations of two-phase
  porous media flows}, Comptes Rendus Mathematique, 347 (2009), pp.~551--554.

\bibitem{Ern}
{\sc A.~Ern, A.~F. Stephansen, and P.~Zunino}, {\em A discontinuous {G}alerkin
  method with weighted averages for advection--diffusion equations with locally
  small and anisotropic diffusivity}, IMA J. Numer. Anal., 29 (2009),
  pp.~235--256.

\bibitem{ErnVo2015}
{\sc A.~Ern and M.~Vohral{\'\i}k}, {\em Polynomial-degree robust a posteriori
  estimates in a unified setting for conforming, nonconforming, discontinuous
  {G}alerkin, and mixed discretizations}, SIAM J. Numer. Anal., 53 (2015),
  pp.~1058--1081.

\bibitem{farina2021cut}
{\sc S.~Farina, S.~Claus, J.~S. Hale, A.~Skupin, and S.~P. Bordas}, {\em A cut
  finite element method for spatially resolved energy metabolism models in
  complex neuro-cell morphologies with minimal remeshing}, A.M.S.E.S., 8
  (2021), pp.~1--32.

\bibitem{guo2021solving}
{\sc R.~Guo}, {\em Solving {P}arabolic {M}oving {I}nterface {P}roblems with
  {D}ynamical {I}mmersed {S}paces on {U}nfitted {M}eshes: {F}ully {D}iscrete
  {A}nalysis}, SIAM J. Numer. Anal., 59 (2021), pp.~797--828.

\bibitem{IRT}
{\sc J.~Haifeng}, {\em An immersed {R}aviart--{T}homas mixed finite element
  method for elliptic interface problems on unfitted meshes}, J. Sci. Comput.,
  (2022).

\bibitem{li1998immersed}
{\sc Z.~Li}, {\em The immersed interface method using a finite element
  formulation}, Appl. Numer. Math., 27 (1998), pp.~253--267.

\bibitem{Nitche}
{\sc J.~Nitsche}, {\em Uber ein {V}ariationsprinzip zur {L}\"osung von
  {D}irichlet-{P}roblemen bei {V}erwendung von {T}eilr\"aumen, die keinen
  {R}andbedingungen unterworfen sind}, Abh. Math. Sem. Univ. Hamburg,  (1971).

\bibitem{odsaeter2017postprocessing}
{\sc L.~H. Ods{\ae}ter, M.~F. Wheeler, T.~Kvamsdal, and M.~G. Larson}, {\em
  Postprocessing of non-conservative flux for compatibility with transport in
  heterogeneous media}, Comput. Methods Appl. Mech. Eng., 315 (2017),
  pp.~799--830.

\bibitem{raviart1977mixed}
{\sc P.~A. Raviart and J.~M. Thomas}, {\em A mixed finite element method for
  second order elliptic problems}, in Mathematical Aspects of the Finite
  Element Method, Lecture Notes in Math., 606, Springer-Verlag, Berlin, 1977.

\bibitem{Ve:09}
{\sc R.~Verf{\"u}rth}, {\em A note on constant-free a posteriori error
  estimates}, SIAM J. Numer. Anal., 47 (2009), pp.~3180--3194.

\bibitem{vohralik2011guaranteed}
{\sc M.~Vohral{\'\i}k}, {\em Guaranteed and fully robust a posteriori error
  estimates for conforming discretizations of diffusion problems with
  discontinuous coefficients}, J. Sci. Comput., 46 (2011), pp.~397--438.

\bibitem{vohralik2013posteriori}
{\sc M.~Vohral{\'\i}k and M.~F. Wheeler}, {\em A posteriori error estimates,
  stopping criteria, and adaptivity for two-phase flows}, Computational
  Geosciences, 17 (2013), pp.~789--812.

\end{thebibliography}
\end{document}